\documentclass[11pt,a4paper]{scrartcl}
\usepackage[english]{babel}
\usepackage[utf8]{inputenc}
\usepackage[T1]{fontenc}
\usepackage{tikz}
\graphicspath{{gnuplot/}}
\usepackage{subcaption}
\usepackage{enumerate}
\usepackage{algorithm}
\usepackage{algorithmic}
\usepackage{hyperref}

\usepackage{mathtools, amsmath, amssymb, amsthm, mathrsfs}
\usepackage{bm} \usepackage{dsfont}

\DeclareMathOperator*{\CV}{CV}
\DeclareMathOperator*{\supp}{supp}

\numberwithin{equation}{section}
\numberwithin{table}{section}
\numberwithin{figure}{section}

\newtheorem{theorem}{Theorem}[section]
\newtheorem{definition}[theorem]{Definition}
\newtheorem{lemma}[theorem]{Lemma}
\newtheorem{remark}[theorem]{Remark}

\title{Concentration Inequalities for Cross-validation in Scattered Data Approximation}
\author{Felix~Bartel \and Ralf~Hielscher}
\date{}

\begin{document} 

\maketitle
\begin{abstract}
Choosing models from a hypothesis space is a frequent task in approximation
theory and inverse problems.
Cross-validation is a classical tool in the learner's repertoire to compare the goodness of fit for different reconstruction models.
Much work has been dedicated to computing this quantity in a fast manner but tackling its theoretical properties occurs to be difficult.
So far, most optimality results are stated in an asymptotic fashion.
In this paper we propose a concentration inequality on the difference of cross-validation score and the risk functional with respect to the squared error.
This gives a pre-asymptotic bound which holds with high probability.
For the assumptions we rely on bounds on the uniform error of the model which allow for a broadly applicable framework.

We support our claims by applying this machinery to Shepard's model, where we are able to determine precise constants of the concentration inequality.
Numerical experiments in combination with fast algorithms indicate the applicability of our results.

\vspace{0.4cm}
\end{abstract}

\section{Introduction} 

The general problem in scattered data approximation is the reconstruction of a function $f\colon\Omega\to Y$ based on discrete samples $\bm z = (z_i)_{i=1}^n = (\bm x_i, f(\bm x_i))_{i=1}^n \in (\Omega\times Y)^n$.
The nodes $\bm x_i$ are independent and identically distributed according to $\rho$ on $\Omega$.
Extensive work has been done to develop reconstruction algorithms $R_h\colon (\Omega \times Y)^n \to Y^\Omega$ which propose candidates for the approximation.
Here, $h$ resembles one of the various methods with possible parameters.
Using multiple reconstruction algorithms $R_h$, $h\in H$ we end up with a hypothesis space $\{R_h(\bm z) : h\in H\} \subset Y^\Omega$.
Even given a precise application, it remains difficult to choose
reconstruction algorithms $R_h$, $h\in H$ which yields the best reconstruction $R_h(\bm z)$ of $f$.

In order to find an optimal $R_h(\bm z)$, $h\in H$, we would like to rank the reconstructions with respect to their goodness of fit.
This is quantified by the \emph{risk functional}.
In this paper we consider the risk functional with respect to the squared loss
\begin{align}\label{eq:l2err}
  \mathcal E(R_h (\bm z))
  = \int_{\Omega} \left|(R_h (\bm z))(\bm x)-f(\bm x)\right|^2\;\mathrm d\rho(\bm x).
\end{align}
Even though this is theoretically appealing we would need to know the underlying distribution $\rho$ and the function $f$ to compute this quantity.
Since this is not the case, we seek for an alternative which only relies on the given data.
The concept which struck our attention is called \emph{cross-validation}, was initially introduced in \cite{GoHeWa79}, and has been widely used since then, cf.\,\cite{TaWe96, BS02, MS00, Rosset09, DPR10}.
The basic idea consists of subdividing the data into a training set and a validation set for estimating the error.
Doing this multiple times we obtain a reasonable estimator for the risk functional.
A special case is where the partitionings seclude single nodes, then the training sets become $\bm z_{-i} \coloneqq (z_1, \dots, z_{i-1}, z_{i+1}, \dots, z_n)$ and the validation sets $\{z_i\}$.
This leads to the so called \emph{leave-one-out cross-validation score}
\begin{align}\label{eq:cv}
  \CV(\bm z, h)
  = \frac 1n \sum_{i=1}^n \left|(R_h (\bm z_{-i})) (\bm x_i) - f(\bm x_i) \right|^2.
\end{align}

An immediate drawback is given by the numerical complexity of computing the $n$ approximations $R_h(\bm z_{-i})$.
However, this is circumvented in many cases with ideas including Monte Carlo approximations \cite{DeGi91}, matrix decomposition methods \cite{Wei07,RoWiBu08}, Krylow space methods \cite{LuHoAn10}, or Fourier analysis \cite{BaHiPo19}.

One is interested in a theoretical foundation of the cross-validation score.
By the Bakushinski\u{\i} veto, cf.\,\cite{Bak84}, we know that there exists a realization of the samples, such that purely data-driven regularization methods have no guarantee for a good approximation without incorporating further information.
One still has propositions about the goodness of the cross-validation score in asymptotic cases, cf.\,\cite{Li86, GKKH02, Lukas06, Gu13}, on average, cf.\,\cite{GoHeWa79, BR08, Becker11}, or by restriction of noise, cf.\,\cite{KN08, KPP18}.

In this paper we bound the difference of cross-validation and risk pre-asymptotically, which supports the choice of cross-validation for model selection.
To circumvent the Bakushinski\u{\i} veto our results will hold with high probability as it is common in learning theory.
We use mild assumptions on the uniform error of the reconstruction algorithm, which allow for a broadly applicable framework.
These bounds improve on the results from \cite[Chapter 8]{GKKH02} in a more general setting.
Other pre-asymptotic results can be found in \cite{KKV11, KLVV13}, where the algorithmic stability, a variance-like concept, of the cross-validation score is examined.

As for the structure of this paper, in Section~\ref{sec:mcdiarmid} we repeat on an extension of McDiarmid's concentration inequality, as it will be of importance later on.
The main part is Section~\ref{sec:framework}, where we present our general framework.
Therefore, we prove in Theorems~\ref{theorem:concentratel2}
and \ref{theorem:concentratecv} concentration inequalities for the risk
functional \eqref{eq:l2err} and the cross-validation score \eqref{eq:cv} with respect to the data $\bm z$.
These concentration inequalities are used to surround the expected values of
the risk functional $\mathcal E(R_h(\cdot))$ and the cross-validation score $\CV(\cdot, h)$ by narrow intervals in which nearly all realizations of these quantities lie.
In Lemma \ref{lemma:connection} we show that the expected values of
$\mathcal E(R_h(\cdot))$ and $\CV(\cdot, h)$ coincide. Eventually, this leads us to our main result in
Theorem~\ref{theorem:connect}, which bounds the difference of risk functional and cross-validation score with high probability and, therefore, justifies the usage of cross-validation for choosing models and parameters.
To exemplify the applicability of our results and reason for the stated conditions to make sense we apply the framework to Shepard's model in Section~\ref{sec:shepard}.
As before, we bound the difference of cross-validation score and risk with high probability, now with precise constants in Theorem~\ref{theorem:shepard}.
We confirm our results with numerical experiments.

 \section{McDiarmid's concentration inequality}\label{sec:mcdiarmid} 

Since it will be of fundamental importance, we dedicate this section to an extension of McDiarmid's concentration inequality.
We consider random variables $\bm X = (X_1, \dots, X_n)$ on a probability space $(\Omega^n, \mathcal A, \mathds P)$.
As usual we denote with
\begin{align*}
  \mathds P\{A|B\} = \frac{\mathds P\{A\cap B\}}{\mathds P\{B\}}
  \quad\text{and}\quad
  \mathds E\{X|B\} = \frac{\mathds E\{\mathds 1_B X\}}{\mathds P\{B\}}
\end{align*}
the conditional probability and expected value, respectively.
To state McDiarmid's theorem we need the following concept.
\begin{definition}
  A function $f\colon\Omega^n\to\mathds R$ is said to be \emph{$\bm c$-bounded} on $\Xi\subset\Omega^n$ for $\bm c = (c_1,\dots,c_n) \in[0,\infty)^n$ if and only if
  \begin{align*}
    |f(\bm x)-f(\bm x')|
    \le d_{\bm c}(\bm x, \bm x')
  \end{align*}
  for all $\bm x = (x_1, \dots, x_n)$ and $\bm x' = (x_1',\dots,x_n') \in\Xi$ where the distance $d_{\bm c}$ is defined by
  \begin{align*}
    d_{\bm c}(\bm x, \bm x') = \sum_{i : x_i \neq x_i'} c_i.
  \end{align*}
\end{definition}

Note, that a function is $\bm c$-bounded if changing a single variable
$x_{i}$, $1\le i\le n$
changes $f(\bm x)$ only by $c_{i}$, i.e.,
\begin{align*}
  \left|f(x_1,\dots,x_n)-f(x_1,\dots, x_{i-1}, x_i', x_{i+1}, \dots ,x_n)\right|
  \le c_i
\end{align*}
for all $(x_1, \dots, x_n)$, $(x_1',\dots,x_n') \in\Xi$.

McDiarmid's inequality, cf.\,\cite{McDiarmid89}, is a generalization of Hoeffding's inequality.
We will not state the original theorem, but an extension from \cite{Combes15}.

\begin{theorem}\label{theorem:combes}
  Let $\bm X = (X_1, \dots, X_n)$ be a vector of independent random variables taking values in a set $\Omega$.
  Furthermore, let $f\colon\Omega^n \to \mathds R$ be $\bm c$-bounded on $\Xi\subset\Omega^n$,
  $ m = \mathds E\left\{f(\bm X) | \bm X\in\Xi\right\}$
  be the expected value of $f$ restricted to $\Xi$, and
  $\gamma = 1-\mathds P\{\bm X\in\Xi\}$
  the probability of $\bm X$ not being in $\Xi$.

  Then we have for $\varepsilon > \gamma\|\bm c\|_1$ the concentration of $f(\bm X)$ around its expected value
  \begin{align*}
    &\mathds P\left\{
      \left|f(\bm X) - m \right|
      > \varepsilon
    \right\}
    \le
    2 \gamma + 2\exp\left(-\frac{2(\varepsilon-\gamma\|\bm c\|_1)^2}{\|\bm c\|_2^2}\right).
  \end{align*}
\end{theorem}

 \section{General framework}\label{sec:framework}

Throughout this section we consider an arbitrary domain $\Omega$ equipped with
some probability measure $\rho$ and a function $f \colon \Omega \to Y$ which
we want to approximate from a finite sampling $\bm{z} =
(x_{i},f(x_{i}))_{i=1}^{n}$. We consider the sampling $\bm{z} \in (\Omega\times Y)^n$ as a
realization of the random variable $\bm Z = (\bm X_{i}, f(\bm X_{i}))_{i=1}^{N}$
with $\bm X_{i}$ being independently and identically $\rho$-distributed random
variables with values in $\Omega$.
This includes the generality of data-driven approximation methods.

The goal of this section is to relate, for an arbitrary approximation operator
$R_{h} \colon (\Omega\times Y)^n \to Y^{\Omega}$, the risk functional
\eqref{eq:l2err} and the cross-validation score \eqref{eq:cv}. This is done in
three steps: First we prove concentration inequalities for the risk functional
and cross-validation score in Theorem~\ref{theorem:concentratel2} and
\ref{theorem:concentratecv}, respectively.  For every reconstruction algorithm
$R_h$, this restricts their values to an interval around their expected values
with high probability as depicted in Figure~\ref{fig:intuition} (a) and (b).
In Lemma~\ref{lemma:connection} we state the connection of these two expected
values.  These three facts allow us to overlap the two concentrations,
cf.\,Figure~\ref{fig:intuition} (c), and lead to Theorem~\ref{theorem:connect}
which is a concentration inequality for the difference of risk functional and
cross-validation score.

\begin{figure}
  \centering
  \begin{subfigure}{0.32\linewidth}
    \centering
    \begingroup
  \makeatletter
  \providecommand\color[2][]{\GenericError{(gnuplot) \space\space\space\@spaces}{Package color not loaded in conjunction with
      terminal option `colourtext'}{See the gnuplot documentation for explanation.}{Either use 'blacktext' in gnuplot or load the package
      color.sty in LaTeX.}\renewcommand\color[2][]{}}\providecommand\includegraphics[2][]{\GenericError{(gnuplot) \space\space\space\@spaces}{Package graphicx or graphics not loaded}{See the gnuplot documentation for explanation.}{The gnuplot epslatex terminal needs graphicx.sty or graphics.sty.}\renewcommand\includegraphics[2][]{}}\providecommand\rotatebox[2]{#2}\@ifundefined{ifGPcolor}{\newif\ifGPcolor
    \GPcolortrue
  }{}\@ifundefined{ifGPblacktext}{\newif\ifGPblacktext
    \GPblacktexttrue
  }{}\let\gplgaddtomacro\g@addto@macro
\gdef\gplbacktext{}\gdef\gplfronttext{}\makeatother
  \ifGPblacktext
\def\colorrgb#1{}\def\colorgray#1{}\else
\ifGPcolor
      \def\colorrgb#1{\color[rgb]{#1}}\def\colorgray#1{\color[gray]{#1}}\expandafter\def\csname LTw\endcsname{\color{white}}\expandafter\def\csname LTb\endcsname{\color{black}}\expandafter\def\csname LTa\endcsname{\color{black}}\expandafter\def\csname LT0\endcsname{\color[rgb]{1,0,0}}\expandafter\def\csname LT1\endcsname{\color[rgb]{0,1,0}}\expandafter\def\csname LT2\endcsname{\color[rgb]{0,0,1}}\expandafter\def\csname LT3\endcsname{\color[rgb]{1,0,1}}\expandafter\def\csname LT4\endcsname{\color[rgb]{0,1,1}}\expandafter\def\csname LT5\endcsname{\color[rgb]{1,1,0}}\expandafter\def\csname LT6\endcsname{\color[rgb]{0,0,0}}\expandafter\def\csname LT7\endcsname{\color[rgb]{1,0.3,0}}\expandafter\def\csname LT8\endcsname{\color[rgb]{0.5,0.5,0.5}}\else
\def\colorrgb#1{\color{black}}\def\colorgray#1{\color[gray]{#1}}\expandafter\def\csname LTw\endcsname{\color{white}}\expandafter\def\csname LTb\endcsname{\color{black}}\expandafter\def\csname LTa\endcsname{\color{black}}\expandafter\def\csname LT0\endcsname{\color{black}}\expandafter\def\csname LT1\endcsname{\color{black}}\expandafter\def\csname LT2\endcsname{\color{black}}\expandafter\def\csname LT3\endcsname{\color{black}}\expandafter\def\csname LT4\endcsname{\color{black}}\expandafter\def\csname LT5\endcsname{\color{black}}\expandafter\def\csname LT6\endcsname{\color{black}}\expandafter\def\csname LT7\endcsname{\color{black}}\expandafter\def\csname LT8\endcsname{\color{black}}\fi
  \fi
    \setlength{\unitlength}{0.0500bp}\ifx\gptboxheight\undefined \newlength{\gptboxheight}\newlength{\gptboxwidth}\newsavebox{\gptboxtext}\fi \setlength{\fboxrule}{0.5pt}\setlength{\fboxsep}{1pt}\begin{picture}(2540.00,1460.00)\gplgaddtomacro\gplbacktext{}\gplgaddtomacro\gplfronttext{}\gplbacktext
    \put(0,0){\includegraphics[width={127.00bp},height={73.00bp}]{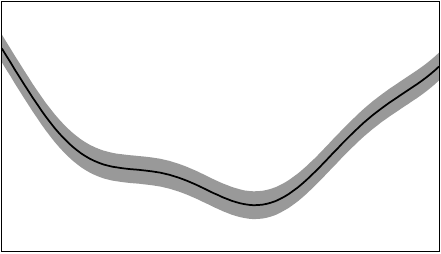}}\gplfronttext
  \end{picture}\endgroup
     \caption{Concentration of the risk functional.}
  \end{subfigure}
  \begin{subfigure}{0.32\linewidth}
    \centering
    \begingroup
  \makeatletter
  \providecommand\color[2][]{\GenericError{(gnuplot) \space\space\space\@spaces}{Package color not loaded in conjunction with
      terminal option `colourtext'}{See the gnuplot documentation for explanation.}{Either use 'blacktext' in gnuplot or load the package
      color.sty in LaTeX.}\renewcommand\color[2][]{}}\providecommand\includegraphics[2][]{\GenericError{(gnuplot) \space\space\space\@spaces}{Package graphicx or graphics not loaded}{See the gnuplot documentation for explanation.}{The gnuplot epslatex terminal needs graphicx.sty or graphics.sty.}\renewcommand\includegraphics[2][]{}}\providecommand\rotatebox[2]{#2}\@ifundefined{ifGPcolor}{\newif\ifGPcolor
    \GPcolortrue
  }{}\@ifundefined{ifGPblacktext}{\newif\ifGPblacktext
    \GPblacktexttrue
  }{}\let\gplgaddtomacro\g@addto@macro
\gdef\gplbacktext{}\gdef\gplfronttext{}\makeatother
  \ifGPblacktext
\def\colorrgb#1{}\def\colorgray#1{}\else
\ifGPcolor
      \def\colorrgb#1{\color[rgb]{#1}}\def\colorgray#1{\color[gray]{#1}}\expandafter\def\csname LTw\endcsname{\color{white}}\expandafter\def\csname LTb\endcsname{\color{black}}\expandafter\def\csname LTa\endcsname{\color{black}}\expandafter\def\csname LT0\endcsname{\color[rgb]{1,0,0}}\expandafter\def\csname LT1\endcsname{\color[rgb]{0,1,0}}\expandafter\def\csname LT2\endcsname{\color[rgb]{0,0,1}}\expandafter\def\csname LT3\endcsname{\color[rgb]{1,0,1}}\expandafter\def\csname LT4\endcsname{\color[rgb]{0,1,1}}\expandafter\def\csname LT5\endcsname{\color[rgb]{1,1,0}}\expandafter\def\csname LT6\endcsname{\color[rgb]{0,0,0}}\expandafter\def\csname LT7\endcsname{\color[rgb]{1,0.3,0}}\expandafter\def\csname LT8\endcsname{\color[rgb]{0.5,0.5,0.5}}\else
\def\colorrgb#1{\color{black}}\def\colorgray#1{\color[gray]{#1}}\expandafter\def\csname LTw\endcsname{\color{white}}\expandafter\def\csname LTb\endcsname{\color{black}}\expandafter\def\csname LTa\endcsname{\color{black}}\expandafter\def\csname LT0\endcsname{\color{black}}\expandafter\def\csname LT1\endcsname{\color{black}}\expandafter\def\csname LT2\endcsname{\color{black}}\expandafter\def\csname LT3\endcsname{\color{black}}\expandafter\def\csname LT4\endcsname{\color{black}}\expandafter\def\csname LT5\endcsname{\color{black}}\expandafter\def\csname LT6\endcsname{\color{black}}\expandafter\def\csname LT7\endcsname{\color{black}}\expandafter\def\csname LT8\endcsname{\color{black}}\fi
  \fi
    \setlength{\unitlength}{0.0500bp}\ifx\gptboxheight\undefined \newlength{\gptboxheight}\newlength{\gptboxwidth}\newsavebox{\gptboxtext}\fi \setlength{\fboxrule}{0.5pt}\setlength{\fboxsep}{1pt}\begin{picture}(2540.00,1460.00)\gplgaddtomacro\gplbacktext{}\gplgaddtomacro\gplfronttext{}\gplbacktext
    \put(0,0){\includegraphics[width={127.00bp},height={73.00bp}]{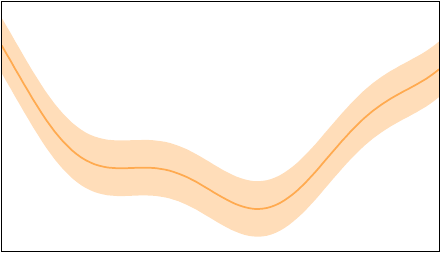}}\gplfronttext
  \end{picture}\endgroup
     \caption{Concentration of the cross-validation score.}
  \end{subfigure}
  \begin{subfigure}{0.32\linewidth}
    \centering
    \begingroup
  \makeatletter
  \providecommand\color[2][]{\GenericError{(gnuplot) \space\space\space\@spaces}{Package color not loaded in conjunction with
      terminal option `colourtext'}{See the gnuplot documentation for explanation.}{Either use 'blacktext' in gnuplot or load the package
      color.sty in LaTeX.}\renewcommand\color[2][]{}}\providecommand\includegraphics[2][]{\GenericError{(gnuplot) \space\space\space\@spaces}{Package graphicx or graphics not loaded}{See the gnuplot documentation for explanation.}{The gnuplot epslatex terminal needs graphicx.sty or graphics.sty.}\renewcommand\includegraphics[2][]{}}\providecommand\rotatebox[2]{#2}\@ifundefined{ifGPcolor}{\newif\ifGPcolor
    \GPcolortrue
  }{}\@ifundefined{ifGPblacktext}{\newif\ifGPblacktext
    \GPblacktexttrue
  }{}\let\gplgaddtomacro\g@addto@macro
\gdef\gplbacktext{}\gdef\gplfronttext{}\makeatother
  \ifGPblacktext
\def\colorrgb#1{}\def\colorgray#1{}\else
\ifGPcolor
      \def\colorrgb#1{\color[rgb]{#1}}\def\colorgray#1{\color[gray]{#1}}\expandafter\def\csname LTw\endcsname{\color{white}}\expandafter\def\csname LTb\endcsname{\color{black}}\expandafter\def\csname LTa\endcsname{\color{black}}\expandafter\def\csname LT0\endcsname{\color[rgb]{1,0,0}}\expandafter\def\csname LT1\endcsname{\color[rgb]{0,1,0}}\expandafter\def\csname LT2\endcsname{\color[rgb]{0,0,1}}\expandafter\def\csname LT3\endcsname{\color[rgb]{1,0,1}}\expandafter\def\csname LT4\endcsname{\color[rgb]{0,1,1}}\expandafter\def\csname LT5\endcsname{\color[rgb]{1,1,0}}\expandafter\def\csname LT6\endcsname{\color[rgb]{0,0,0}}\expandafter\def\csname LT7\endcsname{\color[rgb]{1,0.3,0}}\expandafter\def\csname LT8\endcsname{\color[rgb]{0.5,0.5,0.5}}\else
\def\colorrgb#1{\color{black}}\def\colorgray#1{\color[gray]{#1}}\expandafter\def\csname LTw\endcsname{\color{white}}\expandafter\def\csname LTb\endcsname{\color{black}}\expandafter\def\csname LTa\endcsname{\color{black}}\expandafter\def\csname LT0\endcsname{\color{black}}\expandafter\def\csname LT1\endcsname{\color{black}}\expandafter\def\csname LT2\endcsname{\color{black}}\expandafter\def\csname LT3\endcsname{\color{black}}\expandafter\def\csname LT4\endcsname{\color{black}}\expandafter\def\csname LT5\endcsname{\color{black}}\expandafter\def\csname LT6\endcsname{\color{black}}\expandafter\def\csname LT7\endcsname{\color{black}}\expandafter\def\csname LT8\endcsname{\color{black}}\fi
  \fi
    \setlength{\unitlength}{0.0500bp}\ifx\gptboxheight\undefined \newlength{\gptboxheight}\newlength{\gptboxwidth}\newsavebox{\gptboxtext}\fi \setlength{\fboxrule}{0.5pt}\setlength{\fboxsep}{1pt}\begin{picture}(2540.00,1460.00)\gplgaddtomacro\gplbacktext{}\gplgaddtomacro\gplfronttext{}\gplbacktext
    \put(0,0){\includegraphics[width={127.00bp},height={73.00bp}]{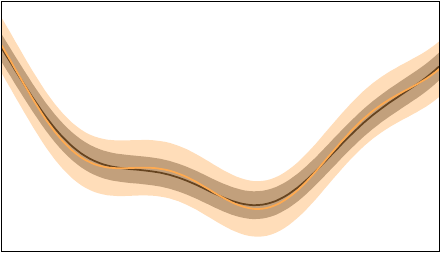}}\gplfronttext
  \end{picture}\endgroup
     \caption{Connection via their expected values. }
  \end{subfigure}
  \caption{Intuition of Theorems~\ref{theorem:concentratel2}, \ref{theorem:concentratecv}, and \ref{theorem:connect}}\label{fig:intuition}.
\end{figure}

Dealing with reconstruction algorithms
$R_h\colon (\Omega \times Y)^n \to Y^\Omega$ in scattered data approximation
settings, there may exist possible realizations $\bm z\in(\Omega\times Y)^n$
of the samples such that we cannot bound the error of the approximation in a
small manner. An example for that would be polynomial interpolation where all
nodes $\bm x_i$ coincide.  To handle these outliers we
define a subset of all samples excluding the outliers without uniform bound
on the reconstruction error.

\begin{definition}\label{definition:alice}
  For a reconstruction method $R_h$ we define a subset of all possible samples
  \begin{align*}
    \Xi
    = \Xi(h, C_1, C_2)
    = \{\bm z\in(\Omega\times Y)^n : \text{(i) and (ii) hold\,}\},
  \end{align*}
  where the two stated conditions are:
  \begin{enumerate}[(i)]
  \item
    The uniform error of the reconstruction $R_h(\bm z)$ is bounded, i.e., for $1\le i\le n$
    \begin{align*}
      \|R_h(\bm z_{-i})-f\|_\infty
      < C_1.
    \end{align*}
  \item
    Changing one node will not do much damage, i.e., for all $\bm x\in\Omega$ we assume for every $1\le i\le n$ the $C_2\mathds 1$-boundedness of $\bm z_{-i} \mapsto R_h(\bm z_{-i})(\bm x)$.
  \end{enumerate}
\end{definition}

\begin{remark}\label{remark:c1c2}
  \begin{enumerate}[(i)]
  \item
    Note that, by applying the triangle inequality, we could use $C_2 \le 2C_1$ and only rely on the first assumption.
    For that reason we will state all results in two ways: one version using only $C_1$ for simplicity and another using both constants to allow for fine-tuning of the bounds.
  \item
    For many reconstruction methods one has a bound on the uniform error in a probabilistic fashion in the form of
    \begin{align*}
      \mathds P\{\|R_h(\bm Z')-f\|_\infty > C_1\}
      \le \gamma
    \end{align*}
      for some small $\gamma$, e.g.\,\cite[Section 6.3 and 6.4]{Stein08} or one of \cite{Kunsch18, LUH19, PU21, KNS21}.
    To extend this to the context of assumption (i), we apply this bound for $\bm Z_{-i}$ and $1\le i\le n$.
    Union bound then gives
    \begin{align*}
      \mathds P\{\bm Z\notin \Xi(h, \varepsilon, 2\varepsilon)\}
      & = \mathds P\{\exists\, 1\le i\le n : \|R_h(\bm Z_{-i})-f\|_\infty > C_1\} \\
      & \le \sum_{i=1}^{n} \mathds P\{ \|R_h(\bm Z_{-i})-f\|_\infty > C_1\} \\
      & \le n\gamma.
    \end{align*}
    For instance, in reconstructing functions via least squares, it has been shown that $\gamma$ decays faster than $1/n$ and the overall probability gets small, cf.\,\cite{PU21}.
    This supports the sanity of the stated set.
  \end{enumerate}
\end{remark}

We now want to show the $\bm c$-boundedness of the risk functional on $\Xi$ in order to apply Theorem~\ref{theorem:combes} for a concentration inequality.

\begin{lemma}\label{lemma:cboundl2} Let $\Xi = \Xi(h, C_1, C_2)$ be the set of samples from Definition~\ref{definition:alice} and $\bm c = 2C_1C_2\mathds 1\in\mathds R^n$.
  Then the risk functionals $\bm z \mapsto \mathcal E(R_h(\bm z_{-i}))$ are $\bm c$-bounded.
\end{lemma}

\begin{proof} We have to check what happens if we change one component.
  For that let $\bm z$ and $\bm z'\in\Xi$ be such that they differ in one sample.
  By the definition of the risk functional and the third binomial formula we have
  \begin{align*}
    &\left| \mathcal E(R_h(\bm z_{-i})) - \mathcal E(R_h(\bm z_{-i}')) \right| \\
    &= \left| \int_{\Omega} |R_h(\bm z_{-i})(\bm x)-f(\bm x)|^2 \;\mathrm d\rho(\bm x)
    - \int_{\Omega} |R_h(\bm z_{-i}')(\bm x)-f(\bm x)|^2 \;\mathrm d\rho(\bm x) \right| \\
    &\le \int_{\Omega}
        |R_h(\bm z_{-i}')(\bm x)-f(\bm x)+R_h(\bm z_{-i})-f(\bm x)|\cdot|R_h(\bm z_{-i}')(\bm x)-R_h(\bm z_{-i})(\bm x)|
      \;\mathrm d\rho(\bm x).
  \end{align*}
  Using property (i) and (ii) of $\Xi$ leads to
  \begin{align*}
    \left| \mathcal E(R_h(\bm z_{-i}')) - \mathcal E(R_h(\bm z_{-i})) \right|
    \le 2 C_1 C_2  \int_{\Omega}\mathrm d\rho(\bm x).
  \end{align*}
  Since $\rho$ is a probability measure the above integral evaluates to one and we obtain the desired constant of $2C_1C_2$.

  In $\mathcal E(R_h(\bm z_{-i}))$ the variable $z_i$ does not occur and, therefore, the corresponding $c_i$ is arbitrary.
  To have a general $\bm c$ for all $1\le i\le n$, we use $c_i = 2C_1C_2$ anyways and obtain the assertion.
\end{proof}

Now we state the theorem on the concentration of the risk functional.

\begin{theorem}\label{theorem:concentratel2} Let $\bm Z = (\bm X_i, f(\bm X_i))_{i=1}^n$ with $\bm X_i$ distributed independent and identically according to $\rho$ on $\Omega$.
  Further, let
  \begin{align*}
    m = \mathds E\{ \mathcal E(R_h(\bm Z_{-i})) | \bm Z\in\Xi \},
  \end{align*}
  be the expected value of the risk functionals $\mathcal E(R_h(\bm Z_{-i}))$ restricted to $\Xi=\Xi(h, C_1, C_2)$ from Definition~\ref{definition:alice}, and
  $\gamma = 1-\mathds P\{\bm Z\in\Xi\}$ the probability of $\bm Z$ not being in $\Xi$.

  Then for $\varepsilon > 2\gamma n C_1 C_2$ and $1\le i\le n$ we obtain the concentration of the risk functionals
  \begin{align*}
    \mathds P\left\{ \left|
    \mathcal E(R_h(\bm Z_{-i}))
    -m \right| > \varepsilon \right\}
    &\le 2\gamma + 2\exp\left(-\left(\frac{\varepsilon}{\sqrt{2n}C_1C_2}-\sqrt{2n}\gamma\right)^2\right)\\
    &\le 2\gamma + 2\exp\left(-\left(\frac{\varepsilon}{\sqrt{8n}C_1^2}-\sqrt{2n}\gamma\right)^2\right).
  \end{align*}
\end{theorem}

\begin{proof} Lemma~\ref{lemma:cboundl2} in combination with Theorem~\ref{theorem:combes}
  yields for $\varepsilon > 2 \gamma n C_1 C_2$ the first inequality
  \begin{align*}
    \mathds P\left\{ \left| \mathds E_{\bm Z'}\left\{ \mathcal E(R_h(\bm Z')) \right\}
    - \mathcal E(R_h(\bm Z)) \right| > \varepsilon \right\}
    \le 2\gamma + 2\exp\left(-\frac{2(\varepsilon-2\gamma n C_1C_2)^2}{4 n C_1^2 C_2^2}\right).
  \end{align*}
  The second inequality is due to Remark~\ref{remark:c1c2} (i).
\end{proof} 

Next, we tackle the related problem with respect to the cross-validation score.
First we take care of its $\bm c$-boundedness on $\Xi$.

\begin{lemma}\label{lemma:cboundcv} Let $\Xi = \Xi(h, C_1, C_2)$ be the set of samples from Definition~\ref{definition:alice} and $\bm c = C_1(C_1/n+2C_2)\mathds 1\in\mathds R^n$.
  Then the cross-validation score $\bm z \mapsto \CV(\bm z, h)$ is $\bm c$-bounded.
\end{lemma}

\begin{proof} We have to check what happens if we change one component.
  For symmetry reasons we only have a look at what happens if we change the first sample.
  Let $\bm z, \bm z'\in\Xi$ be such that
  \begin{align*}
    \bm z
    = \left(z_1, \dots, z_n\right)
    \quad\text{and}\quad
    \bm z'
    = \left(z_1', z_2, \dots, z_n\right).
  \end{align*}
  By the triangle inequality we have
  \begin{align*}
    & |\CV(\bm z, h)-\CV(\bm z', h)| \\
& \le
    \frac 1n \left|
      \left| R_h(\bm z_{-1})(\bm x_1) - f(\bm x_1) \right|^2
      -
      \left| R_h(\bm z_{-1})(\bm x_1') - f(\bm x_1') \right|^2
    \right| \\
    &\phantom= +
      \frac 1n \sum_{i=2}^n \left| R_h(\bm z_{-i})(\bm x_i) - f(\bm x_i) + R_h(\bm z'_{-i})(\bm x_i) - f(\bm x_i) \right|
    \left| R_h(\bm z_{-i})(\bm x_i) - R_h(\bm z'_{-i})(\bm x_i) \right|.
  \end{align*}
  Using the properties of $\Xi$ and $|a^2-b^2| \le \max\{a^2, b^2\}$, we further estimate
  \begin{align*}
    |\CV(\bm z, h)-\CV(\bm z', h)|
    &\le
    \frac{C_1^2
    +
    2(n-1) C_1C_2}{n} \\
    &\le
    C_1(C_1/n + 2 C_2).
  \end{align*}
\end{proof}

The corresponding concentration result looks as follows.

\begin{theorem}\label{theorem:concentratecv} Let $\bm Z = (\bm X_i, f(\bm X_i))_{i=1}^n$ with $\bm X_i$ distributed independent and identically according to $\rho$ on $\Omega$.
  Further, let
  \begin{align*}
    m = \mathds E\{ \CV(\bm Z, h) | \bm Z\in\Xi \},
  \end{align*}
  be the expected value of the cross-validation score $\CV(\bm Z, h)$ restricted to $\Xi=\Xi(h, C_1, C_2)$ from Definition~\ref{definition:alice}, and
  $\gamma = 1-\mathds P\{\bm Z\in\Xi\}$ the probability of $\bm Z$ not being in $\Xi$.

  Then for $\varepsilon > 2\gamma n C_1 C_2+\gamma C_1^2$ we obtain the concentration of the cross-validation score
  \begin{align*}
    \mathds P \left\{ |\CV(\bm Z, h)-m | > \varepsilon \right\}
    & \le
    2\gamma + 2\exp\left(
      -\left(\frac{\sqrt{2}\varepsilon}{C_1(C_1/\sqrt{n}+2\sqrt{n}C_2)}-\sqrt{2n}\gamma\right)^2
    \right) \\
    & \le
    2\gamma + 2\exp\left(
      -\left(\frac{\varepsilon}{3\sqrt{n} C_1^2}-\sqrt{2n}\gamma\right)^2
    \right) \\
  \end{align*}
  where the second inequality holds for $n\ge 5$.
\end{theorem}

\begin{proof} Applying Lemma~\ref{lemma:cboundcv} and Theorem~\ref{theorem:combes} gives the first inequality.
  The second one is obtained by using Remark~\ref{remark:c1c2} (i), $n \ge 5$, and basic calculus.
\end{proof}

Next, we prepare the connection of the two previous theorems by connecting the expected values of the risk functional and the cross-validation score.

\begin{lemma}\label{lemma:connection} The expected value of the risk functional for $n-1$ nodes is equal to the expected value of the cross-validation score for $n$ nodes, i.e.,
  \begin{align*}
    \mathds E_{\bm Z'} \left\{
      \mathcal E( R_h(\bm Z') )
    \right\}
    &= \mathds E_{\bm Z} \left\{
      \CV(\bm Z, h)
    \right\}
  \end{align*}
  for $\bm Z'=(\bm X_i',f(\bm X_i'))_{i=1}^{n-1}$ representing $n-1$ samples and $\bm Z=(\bm X_i,f(\bm X_i))_{i=1}^n$ representing $n$ samples where $\bm X_i$, $\bm X_i'$ are distributed independent and identically according to $\rho$.
\end{lemma}

\begin{proof} Since for all $1\le i\le n$ the $\bm Z_{-i}$ have the same distribution as $\bm Z'$ we write
  \begin{align*}
    \mathds E_{\bm Z'} \left\{
      \mathcal E( R_h(\bm Z') )
    \right\}
    &= \frac 1n\sum_{i=1}^n
    \mathds E_{\bm Z_{-i}}\left\{
      \mathcal E( R_h(\bm Z_{-i}) )
    \right\}.
  \end{align*}
  Instead of using $\mathds E_{\bm Z_{-i}}$, we use $\mathds E_{\bm Z}$ since $Z_i$ does not occur in the corresponding terms
  \begin{align*}
    \mathds E_{\bm Z'} \left\{
      \mathcal E( R_h(\bm Z') )
    \right\}
    = \frac 1n\sum_{i=1}^n
    \mathds E_{\bm Z}\left\{
      \mathcal E( R_h(\bm Z_{-i}) )
    \right\}
    =
    \frac 1n\sum_{i=1}^n
      \mathds E_{\bm Z}\left\{\left|R_h(\bm Z_{-i})(\bm x_i)-f(\bm x_i)\right|^2
    \right\}.
  \end{align*}
  By linearity of the expected value we obtain the assertion
  \begin{align*}
    \mathds E_{\bm Z} \left\{
      \mathcal E( R_h(\bm Z') )
    \right\}
    = \mathds E_{\bm Z}\left\{
    \frac 1n\sum_{i=1}^n
      \left|R_h(\bm Z_{-i})(\bm x_i)-f(\bm x_i)\right|^2
    \right\}
    = \mathds E_{\bm Z}\left\{
      \CV(\bm Z, h)
    \right\}.
  \end{align*}
\end{proof}

Having all the necessary tools, we state a central theorem bringing together risk functional and cross-validation score.

\begin{theorem}\label{theorem:connect} Let $\bm Z = (\bm X_i, f(\bm X_i))_{i=1}^n$ with $\bm X_i$ distributed independent and identically according to $\rho$ on $\Omega$
  and $R_h \colon (\Omega \times Y)^n \to Y^\Omega$ be a reconstruction method.
  Further, let
  \begin{align*}
    M = \sup_{\bm x_1,\dots,\bm x_{n-1}\in\Omega } \|R_h((\bm x_i, f(\bm x_i))_{i=1}^{n-1})\|_\infty
  \end{align*}
  be a uniform bound on the reconstruction for arbitrary nodes and
  $\gamma = 1-\mathds P\{\bm Z\in\Xi\}$ the probability of $\bm Z$ not being in $\Xi = \Xi(h, C_1, C_2)\subseteq (\Omega\times Y)^n$ from Definition~\ref{definition:alice}.

  Then for $\varepsilon > 2\gamma\max\{ 4nC_1C_2+ C_1^2, (M+\|f\|_\infty)^2\}$ we have the concentration bound of the difference of cross-validation score $\CV(\bm Z, h)$ and risk functional $\mathcal E(R_h(\bm Z_{-1}))$
  \begin{align*}
    &\mathds P\left\{
      \left| \CV(\bm Z, h) - \mathcal E(R_h(\bm Z_{-1})) \right|
      > \varepsilon \right\} \\
    &\le
    2\gamma
    + 2\exp\left(-\left(
      \frac{\varepsilon}{\sqrt 2C_1(C_1/\sqrt{n}+4\sqrt{n}C_2)} - \sqrt{2n}\gamma
    \right)^2 \right) \\
    &\le
    2\gamma
    + 2\exp\left(-\left(
      \frac{\varepsilon}{12\sqrt n C_1^2} - \sqrt{2n}\gamma
    \right)^2 \right)
  \end{align*}
  where the second inequality holds for $n\ge 3$.
  In particular, for $\delta>0$, we have with probability larger than $1-2(\gamma+\delta)$
  \begin{align*}
    &|\CV(\bm Z, h) - \mathcal E(R_h(\bm Z_{-1}))| \\
    &\le \max\left\{ 2\gamma(M+\|f\|_\infty)^2, \left(\sqrt 2C_1\left(\frac{C_1}{\sqrt{n}}+4\sqrt{n}C_2\right)\right))\left(\sqrt{2n}\gamma+\sqrt{-\log \delta}\right)\right\} \\
    &\le \max\left\{ 2\gamma(M+\|f\|_\infty)^2, 12\sqrt{n} C_1^2\left(\sqrt{2n}\gamma+\sqrt{-\log \delta}\right)\right\}.
  \end{align*}
\end{theorem}

\begin{proof} By the triangle inequality\footnote{
    One might argue that using triangle inequality with the expected values one looses all information on the specific sample $\bm z$, which worsens the bound.
    However, \cite{BHT21} suggests that $\CV(\cdot, h)$ estimates $\mathds E\{\mathcal E(R_h(\bm Z))\}$ rather than $\mathcal E(R_h(\bm z))$ itself, which reasons for our approach.
  } we have for fixed $\bm z\in(\Omega\times Y)^n$
  \begin{align*}
    &\big|\CV(\bm z, h) - \mathcal E(R_h(\bm z_{-1}))\big| \\
    &\le \left|\CV(\bm z, h) - \mathcal E(R_h(\bm z_{-1}))
    - \mathds E\left\{ \CV(\bm Z, h) - \mathcal E(R_h(\bm Z_{-1})) \middle| \bm Z\in\Xi \right\}\right| \\
    &\phantom= + \left|\mathds E\left\{ \CV(\bm Z, h) - \mathcal E(R_h(\bm Z_{-1})) \middle| \bm Z\in\Xi \right\}\right|.
  \end{align*}
  By Lemma~\ref{lemma:connection} we have $\mathds E\{ \CV(\bm Z, h) - \mathcal E(R_h(\bm Z_{-1})) \} = 0$ and, thus, estimate the second summand by
  \begin{align*}
    &\left|\mathds E\left\{ \CV(\bm Z, h) - \mathcal E(R_h(\bm Z_{-1})) \middle| \bm Z\in\Xi \right\}
    - \mathds E\left\{ \CV(\bm Z, h) - \mathcal E(R_h(\bm Z_{-1})) \right\}\right| \\
    &\le \int_{(\Omega\times Y)^n\setminus\Xi} \left| \CV(\bm z, h) - \mathcal E(R_h(\bm z_{-1}))\right| \;\mathrm d\bm z \\
    &\le \int_{(\Omega\times Y)^n\setminus\Xi}
    \left(M+\|f\|_\infty\right)^2
    \; \mathrm d\bm z \\
    &\le
    \left(M+\|f\|_\infty\right)^2
    \gamma
  \end{align*}
  where the last inequality follows from $\mathds P\{\bm Z\notin\Xi\} \le \gamma$.
  Thus, we obtain
  \begin{align*}
    &\mathds P\left\{
      \left| \CV(\bm Z, h) - \mathcal E(R_h(\bm Z_{-1})) \right|
      > \varepsilon \right\} \\
    &\le \mathds P\left\{ \left|\CV(\bm z, h) - \mathcal E(R_h(\bm z_{-1}))
    - \mathds E\left\{ \CV(\bm Z, h) - \mathcal E(R_h(\bm Z_{-1})) \middle| \bm Z\in\Xi \right\}\right|
    > \frac{\varepsilon}{2}\right\} \\
    &\phantom= +
    \mathds P\left\{
    \left(M+\|f\|_\infty\right)^2
    \gamma
    > \frac{\varepsilon}{2}\right\}.
  \end{align*}
  By the assumption on $\varepsilon$ the latter probability evaluates to zero.

  It is left to bound the first summand.
  Similar to the proofs of Lemmata~\ref{lemma:cboundl2} and \ref{lemma:cboundcv} we will bound the remaining concentration by Theorem~\ref{theorem:combes}.
  For $\bm z$ and $\bm z'\in\Xi$, which differ in one component, we have
  \begin{align*}
    &|\CV(\bm z, h) - \mathcal E(R_h(\bm z_{-1})) -\CV(\bm z', h) + \mathcal E(R_h(\bm z_{-1}'))| \\
    &\le|\CV(\bm z, h)-\CV(\bm z', h)|+|\mathcal E(R_h(\bm z_{-1})) - \mathcal E(R_h(\bm z_{-1}'))| \\
    &\le 4C_1C_2+\frac{C_1^2}{n},
  \end{align*}
  i.e., $\CV(\bm z, h) - \mathcal E(R_h(\bm z_{-1}))$ is $\bm c$-bounded.
  Thus, with Theorem~\ref{theorem:combes} we obtain
  \begin{align*}
    & \mathds P\left\{
      \left|\CV(\bm z, h) - \mathcal E(R_h(\bm z_{-1})) - \mathds E\left\{ \CV(\bm z, h) + \mathcal E(R_h(\bm z_{-1}))\right\}\right| > \varepsilon
    \right\} \\
    & \le 2\gamma+2\exp\left(-\left(
      \frac{\varepsilon}{\sqrt 2C_1(C_1/\sqrt{n}+4\sqrt{n}C_2)} - \sqrt{2n}\gamma
    \right)^2 \right)
  \end{align*}
  for $\varepsilon > 2\gamma( 4nC_1C_2+ C_1^2)$.
\end{proof}

\begin{remark}
  \begin{enumerate}[(i)]
  \item
  If, for a specific reconstruction method $R_h$, we have
  \begin{itemize}
    \item a uniform bound $M$ on the reconstructions $R_h(\bm z)$, $\bm z = (\bm x_i, f(\bm x_i))_{i=1}^n \in(\Omega\times Y)^n$ and
    \item a bound $C_1$ on the reconstructions error of $R_h(\bm z)$ which holds with probability $1-\gamma$,
  \end{itemize}
  then Theorem~\ref{theorem:connect} states, that with slightly smaller probability $1-2(\gamma+\delta)$, computing the cross-validation score $\CV(\bm z, h)$ is the same as computing the risk $\mathcal E(R_h(\bm z))$ up to a small additive constant $\varepsilon$ that can be computed explicitly from $C_1$, $M$, $\gamma$, and $\delta$.
\item
    For now we have a statement for one reconstruction method $R_h$.
    But we easily obtain error guarantees for the parameter $h_{\CV}$ minimizing the cross-validation score $\CV(\bm z, \cdot)$:

    Let $h^\ast$ be the minimizer of $h\mapsto \mathcal E(R_{h}(\bm z))$.
    By using
    \begin{align*}
      &\mathds P\left\{
        \mathcal E(R_{h_{\CV}}(\bm Z_{-1})) - \mathcal E(R_{h^\ast}(\bm Z_{-1}))
        > \varepsilon
      \right\} \\
      &\le
      \mathds P\left\{
        \mathcal E(R_{h_{\CV}}(\bm Z_{-1}))
        - \CV(\bm Z, h_{\CV}) + \CV(\bm Z, h^\ast)
        - \mathcal E(R_{h^\ast}(\bm Z_{-1}))
        > \varepsilon
      \right\} \\
      &\le
      \mathds P\left\{
        \left| \mathcal E(R_{h_{\CV}}(\bm Z_{-1}))
        - \CV(\bm Z, h_{\CV}) \right|
        > \frac{\varepsilon}{2}
      \right\}
      +
      \mathds P\left\{
        \left| \CV(\bm Z, h^\ast)
        - \mathcal E(R_{h^\ast}(\bm Z_{-1})) \right|
        > \frac{\varepsilon}{2}
      \right\}
    \end{align*}
    we apply Theorem~\ref{theorem:connect} twice and have that with high probability minimizing the cross-validation score is just $\varepsilon$ worse in terms of the risk.
  \end{enumerate}
\end{remark}

\begin{remark}
  In order to derive asymptotic rates out of Theorem~\ref{theorem:connect}, we fix the probability $\delta$ and assume that the reconstruction error of $R_h$ decays asymptotically as $C_1\sim n^{-r}$ with probability at least $1-n^{-2r}$.
  Then the difference of cross-validation score $\CV$ and the risk functional $\mathcal E(R_h(\bm z))$ decays like $n^{1/2-2r}$.
\end{remark}

 \section{Application using Shepard's model}\label{sec:shepard}

Since this paper was motivated by \cite[Chapter 8]{GKKH02}, where Shepard's model was used in the context of binary kernels, it seemed natural to start off with this application.
Shepard's model or the Nadaraya-Watson estimator is a special case of moving least squares.
It was introduced in \cite{Nadaraya64, Watson64, Shepard68} and is now-days widely used for solving PDEs \cite{NTV92, BLG94}, manifold learning \cite{SL20}, or computer graphics \cite{SMT06}.
Introductory information about this topic can be found in \cite{Fa07}.

The crucial ingredient in Shepard's model is a, often locally supported, kernel function $K_h$.
Given a sampling $\bm z = (x_i, f(x_i))_{i = 1}^n$ the model has the form
\begin{align}\label{eq:dendermann}
  R_h(\bm z)
  = \frac{\sum_{i=1}^m K_h(\cdot, x_i)f(x_i)}
  {\sum_{i=1}^m K_h(\cdot, x_i)}.
\end{align}
A one-dimensional example for differently localized kernels is shown in Figure~\ref{fig:example}, which emphasizes the importance of the kernel choice.
In this section we propose cross-validation as a method for choosing an optimal kernel and give an explicit error bound for the difference of risk functional \eqref{eq:l2err} and cross-validation score \eqref{eq:cv}.
This is verified with numerical examples.

\begin{figure}
  \centering
  \begin{subfigure}{0.32\linewidth}
    \centering
    \begingroup
  \makeatletter
  \providecommand\color[2][]{\GenericError{(gnuplot) \space\space\space\@spaces}{Package color not loaded in conjunction with
      terminal option `colourtext'}{See the gnuplot documentation for explanation.}{Either use 'blacktext' in gnuplot or load the package
      color.sty in LaTeX.}\renewcommand\color[2][]{}}\providecommand\includegraphics[2][]{\GenericError{(gnuplot) \space\space\space\@spaces}{Package graphicx or graphics not loaded}{See the gnuplot documentation for explanation.}{The gnuplot epslatex terminal needs graphicx.sty or graphics.sty.}\renewcommand\includegraphics[2][]{}}\providecommand\rotatebox[2]{#2}\@ifundefined{ifGPcolor}{\newif\ifGPcolor
    \GPcolortrue
  }{}\@ifundefined{ifGPblacktext}{\newif\ifGPblacktext
    \GPblacktexttrue
  }{}\let\gplgaddtomacro\g@addto@macro
\gdef\gplbacktext{}\gdef\gplfronttext{}\makeatother
  \ifGPblacktext
\def\colorrgb#1{}\def\colorgray#1{}\else
\ifGPcolor
      \def\colorrgb#1{\color[rgb]{#1}}\def\colorgray#1{\color[gray]{#1}}\expandafter\def\csname LTw\endcsname{\color{white}}\expandafter\def\csname LTb\endcsname{\color{black}}\expandafter\def\csname LTa\endcsname{\color{black}}\expandafter\def\csname LT0\endcsname{\color[rgb]{1,0,0}}\expandafter\def\csname LT1\endcsname{\color[rgb]{0,1,0}}\expandafter\def\csname LT2\endcsname{\color[rgb]{0,0,1}}\expandafter\def\csname LT3\endcsname{\color[rgb]{1,0,1}}\expandafter\def\csname LT4\endcsname{\color[rgb]{0,1,1}}\expandafter\def\csname LT5\endcsname{\color[rgb]{1,1,0}}\expandafter\def\csname LT6\endcsname{\color[rgb]{0,0,0}}\expandafter\def\csname LT7\endcsname{\color[rgb]{1,0.3,0}}\expandafter\def\csname LT8\endcsname{\color[rgb]{0.5,0.5,0.5}}\else
\def\colorrgb#1{\color{black}}\def\colorgray#1{\color[gray]{#1}}\expandafter\def\csname LTw\endcsname{\color{white}}\expandafter\def\csname LTb\endcsname{\color{black}}\expandafter\def\csname LTa\endcsname{\color{black}}\expandafter\def\csname LT0\endcsname{\color{black}}\expandafter\def\csname LT1\endcsname{\color{black}}\expandafter\def\csname LT2\endcsname{\color{black}}\expandafter\def\csname LT3\endcsname{\color{black}}\expandafter\def\csname LT4\endcsname{\color{black}}\expandafter\def\csname LT5\endcsname{\color{black}}\expandafter\def\csname LT6\endcsname{\color{black}}\expandafter\def\csname LT7\endcsname{\color{black}}\expandafter\def\csname LT8\endcsname{\color{black}}\fi
  \fi
    \setlength{\unitlength}{0.0500bp}\ifx\gptboxheight\undefined \newlength{\gptboxheight}\newlength{\gptboxwidth}\newsavebox{\gptboxtext}\fi \setlength{\fboxrule}{0.5pt}\setlength{\fboxsep}{1pt}\begin{picture}(2540.00,1460.00)\gplgaddtomacro\gplbacktext{}\gplgaddtomacro\gplfronttext{}\gplbacktext
    \put(0,0){\includegraphics[width={127.00bp},height={73.00bp}]{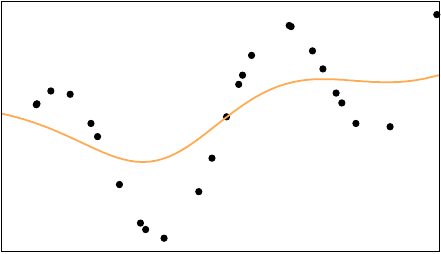}}\gplfronttext
  \end{picture}\endgroup
     \caption{wide support}
  \end{subfigure}
  \begin{subfigure}{0.32\linewidth}
    \centering
    \begingroup
  \makeatletter
  \providecommand\color[2][]{\GenericError{(gnuplot) \space\space\space\@spaces}{Package color not loaded in conjunction with
      terminal option `colourtext'}{See the gnuplot documentation for explanation.}{Either use 'blacktext' in gnuplot or load the package
      color.sty in LaTeX.}\renewcommand\color[2][]{}}\providecommand\includegraphics[2][]{\GenericError{(gnuplot) \space\space\space\@spaces}{Package graphicx or graphics not loaded}{See the gnuplot documentation for explanation.}{The gnuplot epslatex terminal needs graphicx.sty or graphics.sty.}\renewcommand\includegraphics[2][]{}}\providecommand\rotatebox[2]{#2}\@ifundefined{ifGPcolor}{\newif\ifGPcolor
    \GPcolortrue
  }{}\@ifundefined{ifGPblacktext}{\newif\ifGPblacktext
    \GPblacktexttrue
  }{}\let\gplgaddtomacro\g@addto@macro
\gdef\gplbacktext{}\gdef\gplfronttext{}\makeatother
  \ifGPblacktext
\def\colorrgb#1{}\def\colorgray#1{}\else
\ifGPcolor
      \def\colorrgb#1{\color[rgb]{#1}}\def\colorgray#1{\color[gray]{#1}}\expandafter\def\csname LTw\endcsname{\color{white}}\expandafter\def\csname LTb\endcsname{\color{black}}\expandafter\def\csname LTa\endcsname{\color{black}}\expandafter\def\csname LT0\endcsname{\color[rgb]{1,0,0}}\expandafter\def\csname LT1\endcsname{\color[rgb]{0,1,0}}\expandafter\def\csname LT2\endcsname{\color[rgb]{0,0,1}}\expandafter\def\csname LT3\endcsname{\color[rgb]{1,0,1}}\expandafter\def\csname LT4\endcsname{\color[rgb]{0,1,1}}\expandafter\def\csname LT5\endcsname{\color[rgb]{1,1,0}}\expandafter\def\csname LT6\endcsname{\color[rgb]{0,0,0}}\expandafter\def\csname LT7\endcsname{\color[rgb]{1,0.3,0}}\expandafter\def\csname LT8\endcsname{\color[rgb]{0.5,0.5,0.5}}\else
\def\colorrgb#1{\color{black}}\def\colorgray#1{\color[gray]{#1}}\expandafter\def\csname LTw\endcsname{\color{white}}\expandafter\def\csname LTb\endcsname{\color{black}}\expandafter\def\csname LTa\endcsname{\color{black}}\expandafter\def\csname LT0\endcsname{\color{black}}\expandafter\def\csname LT1\endcsname{\color{black}}\expandafter\def\csname LT2\endcsname{\color{black}}\expandafter\def\csname LT3\endcsname{\color{black}}\expandafter\def\csname LT4\endcsname{\color{black}}\expandafter\def\csname LT5\endcsname{\color{black}}\expandafter\def\csname LT6\endcsname{\color{black}}\expandafter\def\csname LT7\endcsname{\color{black}}\expandafter\def\csname LT8\endcsname{\color{black}}\fi
  \fi
    \setlength{\unitlength}{0.0500bp}\ifx\gptboxheight\undefined \newlength{\gptboxheight}\newlength{\gptboxwidth}\newsavebox{\gptboxtext}\fi \setlength{\fboxrule}{0.5pt}\setlength{\fboxsep}{1pt}\begin{picture}(2540.00,1460.00)\gplgaddtomacro\gplbacktext{}\gplgaddtomacro\gplfronttext{}\gplbacktext
    \put(0,0){\includegraphics[width={127.00bp},height={73.00bp}]{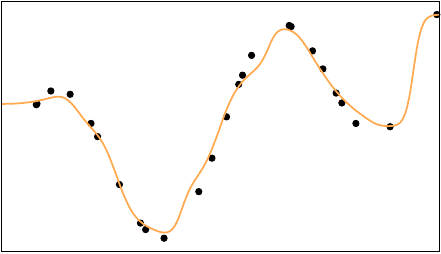}}\gplfronttext
  \end{picture}\endgroup
     \caption{medium support}
  \end{subfigure}
  \begin{subfigure}{0.32\linewidth}
    \centering
    \begingroup
  \makeatletter
  \providecommand\color[2][]{\GenericError{(gnuplot) \space\space\space\@spaces}{Package color not loaded in conjunction with
      terminal option `colourtext'}{See the gnuplot documentation for explanation.}{Either use 'blacktext' in gnuplot or load the package
      color.sty in LaTeX.}\renewcommand\color[2][]{}}\providecommand\includegraphics[2][]{\GenericError{(gnuplot) \space\space\space\@spaces}{Package graphicx or graphics not loaded}{See the gnuplot documentation for explanation.}{The gnuplot epslatex terminal needs graphicx.sty or graphics.sty.}\renewcommand\includegraphics[2][]{}}\providecommand\rotatebox[2]{#2}\@ifundefined{ifGPcolor}{\newif\ifGPcolor
    \GPcolortrue
  }{}\@ifundefined{ifGPblacktext}{\newif\ifGPblacktext
    \GPblacktexttrue
  }{}\let\gplgaddtomacro\g@addto@macro
\gdef\gplbacktext{}\gdef\gplfronttext{}\makeatother
  \ifGPblacktext
\def\colorrgb#1{}\def\colorgray#1{}\else
\ifGPcolor
      \def\colorrgb#1{\color[rgb]{#1}}\def\colorgray#1{\color[gray]{#1}}\expandafter\def\csname LTw\endcsname{\color{white}}\expandafter\def\csname LTb\endcsname{\color{black}}\expandafter\def\csname LTa\endcsname{\color{black}}\expandafter\def\csname LT0\endcsname{\color[rgb]{1,0,0}}\expandafter\def\csname LT1\endcsname{\color[rgb]{0,1,0}}\expandafter\def\csname LT2\endcsname{\color[rgb]{0,0,1}}\expandafter\def\csname LT3\endcsname{\color[rgb]{1,0,1}}\expandafter\def\csname LT4\endcsname{\color[rgb]{0,1,1}}\expandafter\def\csname LT5\endcsname{\color[rgb]{1,1,0}}\expandafter\def\csname LT6\endcsname{\color[rgb]{0,0,0}}\expandafter\def\csname LT7\endcsname{\color[rgb]{1,0.3,0}}\expandafter\def\csname LT8\endcsname{\color[rgb]{0.5,0.5,0.5}}\else
\def\colorrgb#1{\color{black}}\def\colorgray#1{\color[gray]{#1}}\expandafter\def\csname LTw\endcsname{\color{white}}\expandafter\def\csname LTb\endcsname{\color{black}}\expandafter\def\csname LTa\endcsname{\color{black}}\expandafter\def\csname LT0\endcsname{\color{black}}\expandafter\def\csname LT1\endcsname{\color{black}}\expandafter\def\csname LT2\endcsname{\color{black}}\expandafter\def\csname LT3\endcsname{\color{black}}\expandafter\def\csname LT4\endcsname{\color{black}}\expandafter\def\csname LT5\endcsname{\color{black}}\expandafter\def\csname LT6\endcsname{\color{black}}\expandafter\def\csname LT7\endcsname{\color{black}}\expandafter\def\csname LT8\endcsname{\color{black}}\fi
  \fi
    \setlength{\unitlength}{0.0500bp}\ifx\gptboxheight\undefined \newlength{\gptboxheight}\newlength{\gptboxwidth}\newsavebox{\gptboxtext}\fi \setlength{\fboxrule}{0.5pt}\setlength{\fboxsep}{1pt}\begin{picture}(2540.00,1460.00)\gplgaddtomacro\gplbacktext{}\gplgaddtomacro\gplfronttext{}\gplbacktext
    \put(0,0){\includegraphics[width={127.00bp},height={73.00bp}]{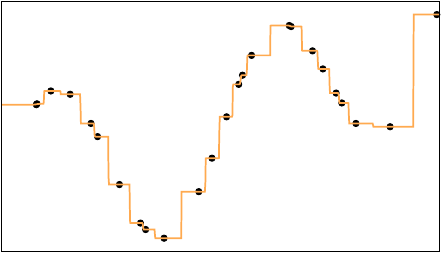}}\gplfronttext
  \end{picture}\endgroup
     \caption{narrow support}
  \end{subfigure}
  \caption{Shepard's model for different widths of the kernel support}\label{fig:example}
\end{figure}

\subsection{Theory} 

For simplicity, we restrict the domain to be the one-dimensional torus $\Omega = \mathds T$ and $Y = \mathds R$.
A common assumption on which we rely is to use positive, radial kernels, i.e.
\begin{align*}
  K_h(x, x')
  = k_h(d(x, x'))
\end{align*}
for $d(\cdot,\cdot)$ being the usual periodic distance on $\mathds T$ and $k_h\colon [0,\infty)\to[0,\infty)$ a family of kernel functions with local support, i.e.,
\begin{align*}
  \supp k_h
  = \overline{\{t\in[0,\infty) : k_h(t) \neq 0\}}
  = [0, 1/h].
\end{align*}

Note, that the range of the function $R_h(\bm z)$ is contained within the convex hull of all $f(x_i)$.
Therefore, for samples $\bm z$ from a bounded function $f\colon\mathds T\to \mathds R$, we have
\begin{align}\label{eq:hui}
  M
  = \sup_{x_1,\dots,x_n\in\Omega} \|R_h((x_i, f(x_i))_{i=1}^n)\|_\infty
  \le \|f\|_\infty.
\end{align}

Deterministic bounds on the approximation error are given in \cite[Chapter 25]{Fa07}.
These are based on the \emph{mesh norm}
\begin{align*}
  \delta_{\{x_1,\dots,x_n\}}
  \coloneqq \max_{x\in\mathds T} \min_{i=1, \dots, n} d(x, x_i).
\end{align*}
For simplicity, we shall use only a simple bound which relies on stronger assumptions compared to \cite[Chapter 25]{Fa07}.
However, this still attains the same order in terms of the mesh norm.

\begin{lemma}\label{lemma:mlsc1} Let $k_h$ be supported on $[0,1/h]$ and $f$ be Lipschitz continuous with constant $L$.
  Furthermore, we assume $ \delta_{\{x_1, \dots, x_n\}} < 1/h $.
  Then
  \begin{align*}
    \|R_h(\bm z)-f\|_\infty
    \le \frac Lh.
  \end{align*}
\end{lemma}

\begin{proof} By the assumption on the mesh norm and the support of $K_h$ we have
  \begin{align*}
    \sum_{i=1}^n K_h(x, x_i) > 0
  \end{align*}
  for all $x\in\mathds T$.
  Thus, we will not divide by zero in the following estimate.
  By the definition of Shepard's method we have
  \begin{align*}
    |R_h(\bm z)(x)-f(x)|
    &= \left| \frac{\sum_{i=1}^m K_h(x,x_i)f(x_i)}{\sum_{i=1}^m K_h(x, x_i)} - f(x) \right| \\
    &\le \frac{\sum_{i=1}^m K_h(x,x_i)|f(x_i)-f(x)|}{\sum_{i=1}^m K_h(x, x_i)}.
  \end{align*}
  Using the Lipschitz condition and the local support we obtain
  \begin{align*}
    |R_h(\bm z)(x)-f(x)|
    &\le L \frac{\sum_{x_i\in[x-1/h,x+1/h]} K_h(x,x_i)|x_i-x|}{\sum_{x_i\in[x-1/h,x+1/h]} K_h(x, x_i)} \\
    &\le \frac Lh \frac{\sum_{x_i\in[x-1/h,x+1/h]} K_h(x,x_i)}{\sum_{x_i\in[x-1/h,x+1/h]} K_h(x, x_i)}
    = \frac Lh.
  \end{align*}
\end{proof}

As we draw samples randomly, we cannot guarantee an upper bound on the mesh norm $\delta_{\{x_1,\dots,x_n\}}$, but aim for a probabilistic result.
Furthermore, in order to bound the approximation errors $C_1$ from Definition~\ref{definition:alice} we actually need a bound for the mesh norms where single nodes are secluded, i.e., for $\delta_{\{x_1,\dots,x_{i-1},x_{i+1},\dots,x_n\}}$ and $1\le i\le n$.
To this end we define
\begin{align}\label{eq:Xi}
  \Xi
  = \left\{ (x_i, f(x_i))_{i=1}^n:
    \delta_{\{x_1, \dots, x_{i-1}, x_{i+1}, \dots, x_n\}} < 1/h \quad\text{for}\quad 1\le i \le n
    \right\}.
\end{align}
By the previous lemma we know, that for samples in $\Xi$ the reconstruction error is bounded by $L/h = C_1$.
With the following lemma we will show that the constructed set is in the paradigm of Definition~\ref{definition:alice} and $\gamma = 1-\mathds P\{\bm z\in\Xi\}$ is close to zero.

\begin{lemma}\label{lemma:biggesthole} For $x_1, \dots, x_n\in\mathds T$ drawn uniformly at random, we have
  \begin{align*}
    \mathds P\left\{
      \exists\,1\le i\le n :
      \delta_{\{x_1, \dots, x_{i-1}, x_{i+1}, \dots, x_n\}} > \tfrac 1h
    \right\}
    \le \sum_{k=1}^{\lfloor h\rfloor} (-1)^{k+1} {n\choose k}\left(1-\frac{k}{2h}\right)^{n-1}.
  \end{align*}
\end{lemma}

\begin{proof}
  The given event on the mesh norm is equivalent to saying the distance of $x_i$ to $x_{i+2}$ will not exceed $1/h$.
  This is certainly fulfilled for nodes where the distance of $x_i$ to $x_{i+1}$ will not exceed $1/(2h)$.
  Therefore,
  \begin{align*}
    \mathds P\left\{
      \exists\,1\le i\le n :
      \delta_{\{x_1, \dots, x_{i-1}, x_{i+1}, \dots, x_n\}} > \tfrac 1h
    \right\}
    \le \mathds P\left\{
      \delta_{\{x_1, \dots, x_n\}} > \tfrac{1}{2h}
    \right\}.
  \end{align*}
  This probability has been calculated in \cite[Theorem~2.1]{Holst80} which gives the assertion.
\end{proof}

\begin{remark}\label{remark:asd}
  \begin{enumerate}[(i)]
  \item
    Note that similar techniques, involving $\varepsilon$-nets, can be applied to obtain results for more general domains, cf.\,\cite{GLPT07}.
  \item
    Figure~\ref{fig:meshnorm} depicts the probability of all mesh norms $\delta_{\{x_1,\dots,x_{i-1},x_{i+1},\dots,x_n\}}$, $1\le i\le n$ being bigger than $1/h$ for $n = 10\,000$ nodes estimated from numerical experiments.
    The critical point is around $1\,000$, where the probability increases away from zero.
    The theoretical bound from Lemma~\ref{lemma:biggesthole} is not optimal and has its critical point around $700$.
  \item
    The binomial bound in Lemma~\ref{lemma:biggesthole} is difficult to evaluate.
    In \cite{Devroye81} it was show that for $n\to\infty$ it converges to the Gumbel distribution, i.e.,
    \begin{align*}
      \sum_{k=1}^{\lfloor h\rfloor} (-1)^{k+1} {n\choose k}\left(1-\frac{k}{2h}\right)^{n-1}
      \to 1-\exp\left(-n\exp\left(-\frac{n}{2h}\right)\right).
    \end{align*}
    In Figure~\ref{fig:meshnorm} we see that, already for $10\,000$ nodes, we are very close to this Gumbel distribution.
  \end{enumerate}
\end{remark}

\begin{figure}
  \centering
  \begingroup
  \makeatletter
  \providecommand\color[2][]{\GenericError{(gnuplot) \space\space\space\@spaces}{Package color not loaded in conjunction with
      terminal option `colourtext'}{See the gnuplot documentation for explanation.}{Either use 'blacktext' in gnuplot or load the package
      color.sty in LaTeX.}\renewcommand\color[2][]{}}\providecommand\includegraphics[2][]{\GenericError{(gnuplot) \space\space\space\@spaces}{Package graphicx or graphics not loaded}{See the gnuplot documentation for explanation.}{The gnuplot epslatex terminal needs graphicx.sty or graphics.sty.}\renewcommand\includegraphics[2][]{}}\providecommand\rotatebox[2]{#2}\@ifundefined{ifGPcolor}{\newif\ifGPcolor
    \GPcolortrue
  }{}\@ifundefined{ifGPblacktext}{\newif\ifGPblacktext
    \GPblacktexttrue
  }{}\let\gplgaddtomacro\g@addto@macro
\gdef\gplbacktext{}\gdef\gplfronttext{}\makeatother
  \ifGPblacktext
\def\colorrgb#1{}\def\colorgray#1{}\else
\ifGPcolor
      \def\colorrgb#1{\color[rgb]{#1}}\def\colorgray#1{\color[gray]{#1}}\expandafter\def\csname LTw\endcsname{\color{white}}\expandafter\def\csname LTb\endcsname{\color{black}}\expandafter\def\csname LTa\endcsname{\color{black}}\expandafter\def\csname LT0\endcsname{\color[rgb]{1,0,0}}\expandafter\def\csname LT1\endcsname{\color[rgb]{0,1,0}}\expandafter\def\csname LT2\endcsname{\color[rgb]{0,0,1}}\expandafter\def\csname LT3\endcsname{\color[rgb]{1,0,1}}\expandafter\def\csname LT4\endcsname{\color[rgb]{0,1,1}}\expandafter\def\csname LT5\endcsname{\color[rgb]{1,1,0}}\expandafter\def\csname LT6\endcsname{\color[rgb]{0,0,0}}\expandafter\def\csname LT7\endcsname{\color[rgb]{1,0.3,0}}\expandafter\def\csname LT8\endcsname{\color[rgb]{0.5,0.5,0.5}}\else
\def\colorrgb#1{\color{black}}\def\colorgray#1{\color[gray]{#1}}\expandafter\def\csname LTw\endcsname{\color{white}}\expandafter\def\csname LTb\endcsname{\color{black}}\expandafter\def\csname LTa\endcsname{\color{black}}\expandafter\def\csname LT0\endcsname{\color{black}}\expandafter\def\csname LT1\endcsname{\color{black}}\expandafter\def\csname LT2\endcsname{\color{black}}\expandafter\def\csname LT3\endcsname{\color{black}}\expandafter\def\csname LT4\endcsname{\color{black}}\expandafter\def\csname LT5\endcsname{\color{black}}\expandafter\def\csname LT6\endcsname{\color{black}}\expandafter\def\csname LT7\endcsname{\color{black}}\expandafter\def\csname LT8\endcsname{\color{black}}\fi
  \fi
    \setlength{\unitlength}{0.0500bp}\ifx\gptboxheight\undefined \newlength{\gptboxheight}\newlength{\gptboxwidth}\newsavebox{\gptboxtext}\fi \setlength{\fboxrule}{0.5pt}\setlength{\fboxsep}{1pt}\definecolor{tbcol}{rgb}{1,1,1}\begin{picture}(3400.00,2260.00)\gplgaddtomacro\gplbacktext{\csname LTb\endcsname \put(441,617){\makebox(0,0)[r]{\strut{}$0$}}\csname LTb\endcsname \put(441,1332){\makebox(0,0)[r]{\strut{}$0.5$}}\csname LTb\endcsname \put(441,2046){\makebox(0,0)[r]{\strut{}$1$}}\csname LTb\endcsname \put(1047,424){\makebox(0,0){\strut{}$500$}}\csname LTb\endcsname \put(2066,424){\makebox(0,0){\strut{}$1500$}}\csname LTb\endcsname \put(3085,424){\makebox(0,0){\strut{}$2500$}}}\gplgaddtomacro\gplfronttext{\csname LTb\endcsname \put(1812,135){\makebox(0,0){\strut{}$h$}}}\gplbacktext
    \put(0,0){\includegraphics[width={170.00bp},height={113.00bp}]{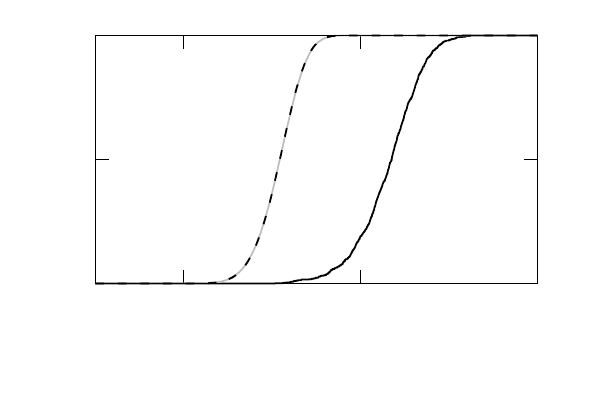}}\gplfronttext
  \end{picture}\endgroup
   \caption{The probability of all mesh norms $\delta_{\{x_1,\dots,x_{i-1},x_{i+1},\dots,x_n\}}$, $1\le i\le n$ being bigger than $1/h$ for $n = 10\,000$ nodes.
    The solid line displays the numerical estimates from $1\,000$ experiments, the dashed line the upper bound from Lemma~\ref{lemma:biggesthole} and the gray line the asymptotic behaviour from Remark~\ref{remark:asd}~(ii).
    }
  \label{fig:meshnorm}
\end{figure}

Now we have the necessary constants: the bound on the reconstruction $M$ and the uniform bound on the reconstruction error $C_1$ with its fail probability $\gamma$ and are able to use the machinery of Section~\ref{sec:framework} to concentrate the difference of risk functional and cross-validation score.

\begin{theorem}\label{theorem:shepard} Let $\bm Z = ((X_1, f(X_1)), \dots, (X_n, f(X_n))$ represent $n$ samples from a function $f\colon \mathds T\to\mathds R$ with Lipschitz constant $L$, and $R_h(\bm Z)$ the reconstruction via Shepard's model, defined by \eqref{eq:dendermann}, where the kernel $k_h$ is supported on $[0, 1/h]$.
  Further, let
  \begin{align*}
    \gamma
    = \sum_{k=1}^{\lfloor h\rfloor} (-1)^{k+1} {n\choose k}\left(1-\frac{k}{2h}\right)^{n-1}
    \quad\text{and}\quad
    \varepsilon > 2\gamma\max\{(4n+1)L^2/h^2, 4\|f\|_\infty^2\}.
  \end{align*}
  Then we have the concentration bound of the difference of cross-validation score $\CV(\bm Z, h)$ and risk functional $\mathcal E(R_h(\bm Z))$
  \begin{align*}
    \mathds P\left\{
      \left| \CV(\bm Z, h) - \mathcal E(R_h(\bm Z_{-1})) \right|
      > \varepsilon \right\}
    \le 2\gamma
    +2\exp\left(-\left(
      \frac{h^2\varepsilon}{12\sqrt n L^2}-\sqrt{2n}\gamma
    \right)^2\right).
  \end{align*}
  In particular for $\delta>0$ we have with probability larger than $1-2(\gamma+\delta)$
  \begin{align*}
    |\CV(\bm Z, h) - \mathcal E(R_h(\bm Z_{-1}))|
    \le \max\left\{ 4\gamma\|f\|_\infty^2, \frac{12\sqrt{n} L^2}{h^2}\left(\sqrt{2n}\gamma+\sqrt{-\log \delta}\right)\right\}
  \end{align*}
\end{theorem}

\begin{proof} By equation~\eqref{eq:hui} we have $M \le \|f\|_\infty$.
  With $\Xi$ as in \eqref{eq:Xi} we have by Lemmata~\ref{lemma:mlsc1} and \ref{lemma:biggesthole}
  \begin{align*}
    C_1 \le \frac Lh
    \quad\text{and}\quad
    \gamma \le \sum_{k=1}^{\lfloor h\rfloor} (-1)^{k+1} {n\choose k}\left(1-\frac{k}{2h}\right)^{n-1}.
  \end{align*}
  Using these constants in Theorem~\ref{theorem:connect} gives the assertion.
\end{proof}

\begin{remark}
  In order to interpret the error bounds in Theorem~\ref{theorem:shepard} asymptotically for $n\to\infty$ we have to fix the desired probability $\delta$.
  Furthermore, we relate the kernel support $1/h$ and the number of samples $n$ via $h = \alpha\cdot n$.
  By Lemma~\ref{lemma:biggesthole} and Remark~\ref{remark:asd} we approximate the fail probability by 
  \begin{align*}
    \gamma\lessapprox\exp(-\mathrm e^{-1/\alpha} n).
  \end{align*}
  Inserting this bound into Theorem~\ref{theorem:shepard}, we obtain with probability $1-2(\exp(-\mathrm e^{-1/\alpha} n)+\delta)$ that
  \begin{align*}
    |\CV(\bm Z, h) - \mathcal E(R_h(\bm Z_{-1}))|
    \sim \max\left\{ \exp(-\mathrm e^{-1/\alpha}n), \frac{\exp(-\mathrm e^{-1/\alpha}n)}{n}+n^{-3/2}\right\}
    \lesssim n^{-3/2}.
  \end{align*}
\end{remark}

\begin{remark}
  The trade off between the constants $C_1, C_2$, and the fail probability $\gamma$ is controlled by the construction of $\Xi$.
  In general, a larger set $\Xi$ leads to a smaller fail probability $\gamma$ but worse constants $C_1$ and $C_2$.

  In the extreme case we have $\gamma = 0$ and $\Xi$ consists of all possible data realizations, i.e., $\Xi = \{(x_i, f(x_i))_{i=1}^n : x_1, \dots, x_n \in\Omega\}$.
  Then we have the bound $C_1 = 2\|f\|_\infty$ as in equation~\eqref{eq:hui}.
  For the specific case of binary kernels, the estimate $C_2 \sim 1/n$ can be found in \cite[page 118]{GKKH02} (with slight adaptions, as there is an individual $C_2$ for every node $x_i$ and one more assumption).
  With that, analogously to Theorem~\ref{theorem:shepard}, we obtain with probability $1-2\delta$
  \begin{align*}
    |\CV(\bm Z, h) - \mathcal E(R_h(\bm Z_{-1}))|
    \sim \max\left\{ 0, 0+n^{-1/2}\right\}
    \lesssim n^{-1/2}.
  \end{align*}
  So, ignoring the restriction to binary kernels, the cost of improving to $\gamma = 0$ is loosing one order in $n$.
  This reasons for the construction of $\Xi$ being a real subset of all possible data realizations.
\end{remark}

\subsection{Implementation} 

Before presenting our numerical experiments in Section~\ref{ss:numerics}, we give a brief discussion on the computational complexity of evaluating the model \eqref{eq:dendermann} as well as computing the cross-validation score $\CV(\bm z, h)$.
Evaluating the model \eqref{eq:dendermann} in nodes $\tilde x_1,\dots,\tilde x_{\tilde n}$ needs two matrix-vector multiplications with
\begin{align*}
  [K_h(x_i, \tilde x_j)]_{i = 1,\dots,\tilde n,\; j = 1,\dots,n}.
\end{align*}
In \cite{FaZh03} a method is proposed to compute \eqref{eq:dendermann} in a fast manner using the nonequispaced fast Fourier transform \cite{nfft3} which works for global kernels.
Since we are dealing with locally supported kernels, we use sparse matrices for an efficient implementation.
To compute the cross-validation score we need to compute $R_h(\bm z_{-i})(x_i)$ for $1\le i \le n$.
To circumvent setting up $n$ models we use the following trick.
For fixed $i$, we obtain
\begin{align*}
  r_i
  \coloneqq R_h(\bm z_{-i}, h)(x_i)
  &= \frac{\sum_{j\in\{1,\dots,n\}\setminus\{i\}}K_h(x_j, x_i)f(x_j)}{\sum_{j\in\{1,\dots,n\}\setminus\{i\}}K_h(x_j, x_i)} \\
  &= \frac{\sum_{j=1}^{n} K_h(x_j, x_i)f(x_j)-k_h(0)f(x_i)}{\sum_{j=1}^{n}K_h(x_j, x_i)-k_h(0)}.
\end{align*}
This favors the Algorithm~\ref{algo:fcv} to compute the cross-validation score.

\begin{algorithm}[ht]
  \medskip

  \textbf{Input:} data $\bm z\in(\mathds T\times \mathds R)^n$
  \medskip

  \textbf{Output:} cross-validation score $\CV(\bm z, h)$
  \medskip

  \begin{algorithmic}[1]
    \FOR{$i=1, \dots, n$}
    \STATE{
      $n_i \leftarrow \sum_{j=1}^n K_h(x_j, x_i)f(x_j)$ \hfill\COMMENT{numerator of Shepard's model}
    }
    \STATE{
      $d_i \leftarrow \sum_{j=1}^n K_h(x_j, x_i)$ \hfill\COMMENT{denominator of Shepard's model}
    }
    \ENDFOR
    \FOR{$i=1, \dots, n$}
      \STATE{
        $r_i = (n_i-k_h(0)f(x_i))/(d_i-k_h(0))$
      }
    \ENDFOR
    \STATE{
      $\CV(\bm z, h) = \frac 1n \sum_{i=1}^n |r_i-f(x_i)|^2$
    }
  \end{algorithmic}
  \caption{Fast cross-validation for Shepard's model}
  \label{algo:fcv}
\end{algorithm}

In terms of complexity we obtain the same as for evaluating the model, namely, two matrix-vector multiplications.

\subsection{Numerics}\label{ss:numerics} 

To exemplify our findings, we present some numerical experiments.
We use the function $f(x) = \sqrt{2}\sin(2\pi x)$ on $\mathds T$ with $\|f\|_{\mathrm L_2(\mathds T)} = 1$, $\|f\|_\infty = \sqrt{2}$, and Lipschitz constant $L = \sqrt 2$.
Further, we choose the simple hat kernel function
\begin{align*}
  k_h(t) = \max\{0, 1-ht\}.
\end{align*}
We then repeat the following experiment $1\,000$ times for $50$ different parameters $h$:
\begin{enumerate}[(i)]
\item
  Choose $n = 10\,000$ uniformly random nodes $x_1, \dots, x_n$.
\item
  Compute function samples $\bm z = (x_i, f(x_i))_{i=1}^n$.
\item
  Compute the reconstruction $R_h(\bm z)$ and approximate the risk $\mathcal E(R_h(\bm z))$ by using evaluations in equispaced nodes.
\item
  Compute the cross-validation score $\CV(\bm z, h)$ via Algorithm~\ref{algo:fcv}.
\end{enumerate}
Figure~\ref{fig:bound} (a) shows the risk $\mathcal E(R_h(\bm z))$ and (b) the cross-validation $\CV(\bm z, h)$ score for every experiment as a single dot.
We observe, that both graphics resemble each other quite nicely.
Both, the risk $\mathcal E(R_h(\bm z))$ and the cross-validation $\CV(\bm z, h)$, increase for small $h$ and become increasingly unstable for $h>1500$ as the support of $K_h$ gets too small.

In order to summarize the statistical behaviour we depicted in Figure~\ref{fig:bound} (c) the corresponding mean values and the intervals where $90\%$ of the outcomes landed with respect the parameter $h$.
The dashed lines depict our concentration bounds from Theorems~\ref{theorem:concentratel2} and \ref{theorem:concentratecv}.
Setting the probability to 0.9, as in the experiment, we obtain the concentration bounds
\begin{align}\label{eq:boundeps}
  \varepsilon
  \le \alpha\frac{L^2}{h^2}\left(\sqrt 2 n \gamma + \sqrt{-n\log\left(\frac p2 - \gamma\right)}\right)
\end{align}
for the risk functional with $\alpha = \sqrt 8$ and the cross-validation score with $\alpha = 3$.
For the fail probability $\gamma$ we used the numerical estimate from Remark~\ref{remark:asd} instead of the theoretical value from Lemma~\ref{lemma:biggesthole}.

Finally, we depicted in Figure~\ref{fig:bound} (d) the $90\%$-quantile of the difference
between the cross-validation score and risk functional.
It illustrates that the risk functional and the cross-validation score coincide very well in the parameter region $200<h<1500$ of interest.
Our main result in Theorem~\ref{theorem:shepard} confirms this by a theoretical bound on this $90\%$-quantile.
The theoretical bound has exactly the form \eqref{eq:boundeps} with $\alpha=12$ and is plotted as a dashed line.

\begin{figure}
  \begin{subfigure}{0.50\linewidth}
    \centering
    \begingroup
  \makeatletter
  \providecommand\color[2][]{\GenericError{(gnuplot) \space\space\space\@spaces}{Package color not loaded in conjunction with
      terminal option `colourtext'}{See the gnuplot documentation for explanation.}{Either use 'blacktext' in gnuplot or load the package
      color.sty in LaTeX.}\renewcommand\color[2][]{}}\providecommand\includegraphics[2][]{\GenericError{(gnuplot) \space\space\space\@spaces}{Package graphicx or graphics not loaded}{See the gnuplot documentation for explanation.}{The gnuplot epslatex terminal needs graphicx.sty or graphics.sty.}\renewcommand\includegraphics[2][]{}}\providecommand\rotatebox[2]{#2}\@ifundefined{ifGPcolor}{\newif\ifGPcolor
    \GPcolortrue
  }{}\@ifundefined{ifGPblacktext}{\newif\ifGPblacktext
    \GPblacktexttrue
  }{}\let\gplgaddtomacro\g@addto@macro
\gdef\gplbacktext{}\gdef\gplfronttext{}\makeatother
  \ifGPblacktext
\def\colorrgb#1{}\def\colorgray#1{}\else
\ifGPcolor
      \def\colorrgb#1{\color[rgb]{#1}}\def\colorgray#1{\color[gray]{#1}}\expandafter\def\csname LTw\endcsname{\color{white}}\expandafter\def\csname LTb\endcsname{\color{black}}\expandafter\def\csname LTa\endcsname{\color{black}}\expandafter\def\csname LT0\endcsname{\color[rgb]{1,0,0}}\expandafter\def\csname LT1\endcsname{\color[rgb]{0,1,0}}\expandafter\def\csname LT2\endcsname{\color[rgb]{0,0,1}}\expandafter\def\csname LT3\endcsname{\color[rgb]{1,0,1}}\expandafter\def\csname LT4\endcsname{\color[rgb]{0,1,1}}\expandafter\def\csname LT5\endcsname{\color[rgb]{1,1,0}}\expandafter\def\csname LT6\endcsname{\color[rgb]{0,0,0}}\expandafter\def\csname LT7\endcsname{\color[rgb]{1,0.3,0}}\expandafter\def\csname LT8\endcsname{\color[rgb]{0.5,0.5,0.5}}\else
\def\colorrgb#1{\color{black}}\def\colorgray#1{\color[gray]{#1}}\expandafter\def\csname LTw\endcsname{\color{white}}\expandafter\def\csname LTb\endcsname{\color{black}}\expandafter\def\csname LTa\endcsname{\color{black}}\expandafter\def\csname LT0\endcsname{\color{black}}\expandafter\def\csname LT1\endcsname{\color{black}}\expandafter\def\csname LT2\endcsname{\color{black}}\expandafter\def\csname LT3\endcsname{\color{black}}\expandafter\def\csname LT4\endcsname{\color{black}}\expandafter\def\csname LT5\endcsname{\color{black}}\expandafter\def\csname LT6\endcsname{\color{black}}\expandafter\def\csname LT7\endcsname{\color{black}}\expandafter\def\csname LT8\endcsname{\color{black}}\fi
  \fi
    \setlength{\unitlength}{0.0500bp}\ifx\gptboxheight\undefined \newlength{\gptboxheight}\newlength{\gptboxwidth}\newsavebox{\gptboxtext}\fi \setlength{\fboxrule}{0.5pt}\setlength{\fboxsep}{1pt}\begin{picture}(4520.00,2820.00)\gplgaddtomacro\gplbacktext{\csname LTb\endcsname \put(616,1138){\makebox(0,0)[r]{\strut{}$10^{-6}$}}\csname LTb\endcsname \put(616,2109){\makebox(0,0)[r]{\strut{}$10^{-4}$}}\csname LTb\endcsname \put(1521,448){\makebox(0,0){\strut{}$500$}}\csname LTb\endcsname \put(3282,448){\makebox(0,0){\strut{}$1500$}}}\gplgaddtomacro\gplfronttext{\csname LTb\endcsname \put(2445,142){\makebox(0,0){\strut{}$h$}}}\gplbacktext
    \put(0,0){\includegraphics[width={226.00bp},height={141.00bp}]{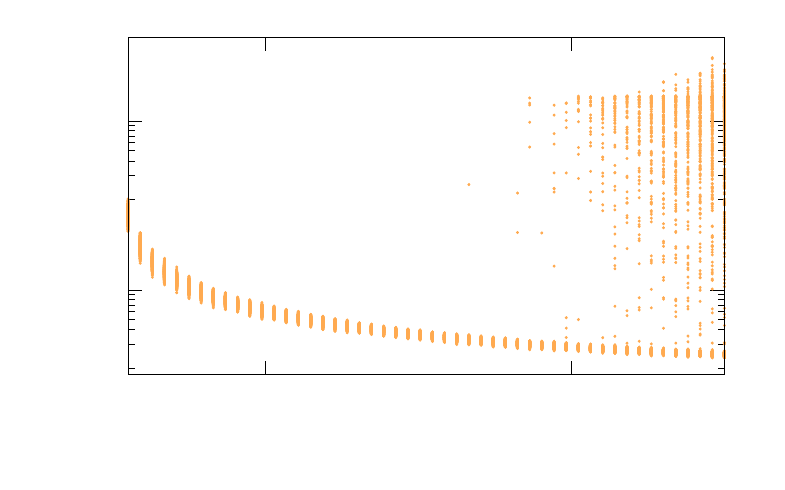}}\gplfronttext
  \end{picture}\endgroup
     \caption{Cross-validation score $\CV(\bm z, h)$ for every experiment.}
  \end{subfigure}
  \begin{subfigure}{0.50\linewidth}
    \centering
    \begingroup
  \makeatletter
  \providecommand\color[2][]{\GenericError{(gnuplot) \space\space\space\@spaces}{Package color not loaded in conjunction with
      terminal option `colourtext'}{See the gnuplot documentation for explanation.}{Either use 'blacktext' in gnuplot or load the package
      color.sty in LaTeX.}\renewcommand\color[2][]{}}\providecommand\includegraphics[2][]{\GenericError{(gnuplot) \space\space\space\@spaces}{Package graphicx or graphics not loaded}{See the gnuplot documentation for explanation.}{The gnuplot epslatex terminal needs graphicx.sty or graphics.sty.}\renewcommand\includegraphics[2][]{}}\providecommand\rotatebox[2]{#2}\@ifundefined{ifGPcolor}{\newif\ifGPcolor
    \GPcolortrue
  }{}\@ifundefined{ifGPblacktext}{\newif\ifGPblacktext
    \GPblacktexttrue
  }{}\let\gplgaddtomacro\g@addto@macro
\gdef\gplbacktext{}\gdef\gplfronttext{}\makeatother
  \ifGPblacktext
\def\colorrgb#1{}\def\colorgray#1{}\else
\ifGPcolor
      \def\colorrgb#1{\color[rgb]{#1}}\def\colorgray#1{\color[gray]{#1}}\expandafter\def\csname LTw\endcsname{\color{white}}\expandafter\def\csname LTb\endcsname{\color{black}}\expandafter\def\csname LTa\endcsname{\color{black}}\expandafter\def\csname LT0\endcsname{\color[rgb]{1,0,0}}\expandafter\def\csname LT1\endcsname{\color[rgb]{0,1,0}}\expandafter\def\csname LT2\endcsname{\color[rgb]{0,0,1}}\expandafter\def\csname LT3\endcsname{\color[rgb]{1,0,1}}\expandafter\def\csname LT4\endcsname{\color[rgb]{0,1,1}}\expandafter\def\csname LT5\endcsname{\color[rgb]{1,1,0}}\expandafter\def\csname LT6\endcsname{\color[rgb]{0,0,0}}\expandafter\def\csname LT7\endcsname{\color[rgb]{1,0.3,0}}\expandafter\def\csname LT8\endcsname{\color[rgb]{0.5,0.5,0.5}}\else
\def\colorrgb#1{\color{black}}\def\colorgray#1{\color[gray]{#1}}\expandafter\def\csname LTw\endcsname{\color{white}}\expandafter\def\csname LTb\endcsname{\color{black}}\expandafter\def\csname LTa\endcsname{\color{black}}\expandafter\def\csname LT0\endcsname{\color{black}}\expandafter\def\csname LT1\endcsname{\color{black}}\expandafter\def\csname LT2\endcsname{\color{black}}\expandafter\def\csname LT3\endcsname{\color{black}}\expandafter\def\csname LT4\endcsname{\color{black}}\expandafter\def\csname LT5\endcsname{\color{black}}\expandafter\def\csname LT6\endcsname{\color{black}}\expandafter\def\csname LT7\endcsname{\color{black}}\expandafter\def\csname LT8\endcsname{\color{black}}\fi
  \fi
    \setlength{\unitlength}{0.0500bp}\ifx\gptboxheight\undefined \newlength{\gptboxheight}\newlength{\gptboxwidth}\newsavebox{\gptboxtext}\fi \setlength{\fboxrule}{0.5pt}\setlength{\fboxsep}{1pt}\begin{picture}(4520.00,2820.00)\gplgaddtomacro\gplbacktext{\csname LTb\endcsname \put(616,1041){\makebox(0,0)[r]{\strut{}$10^{-6}$}}\csname LTb\endcsname \put(616,1818){\makebox(0,0)[r]{\strut{}$10^{-4}$}}\csname LTb\endcsname \put(616,2595){\makebox(0,0)[r]{\strut{}$10^{-2}$}}\csname LTb\endcsname \put(1521,448){\makebox(0,0){\strut{}$500$}}\csname LTb\endcsname \put(3282,448){\makebox(0,0){\strut{}$1500$}}}\gplgaddtomacro\gplfronttext{\csname LTb\endcsname \put(2445,142){\makebox(0,0){\strut{}$h$}}}\gplbacktext
    \put(0,0){\includegraphics[width={226.00bp},height={141.00bp}]{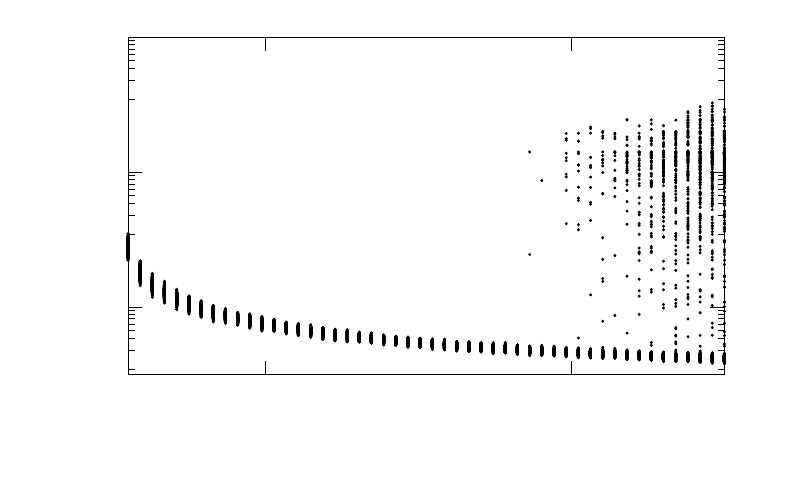}}\gplfronttext
  \end{picture}\endgroup
     \caption{Risk functional $\mathcal E(R_h(\bm z))$ for every experiment.}
  \end{subfigure}
  \begin{subfigure}{0.50\linewidth}
    \centering
    \begingroup
  \makeatletter
  \providecommand\color[2][]{\GenericError{(gnuplot) \space\space\space\@spaces}{Package color not loaded in conjunction with
      terminal option `colourtext'}{See the gnuplot documentation for explanation.}{Either use 'blacktext' in gnuplot or load the package
      color.sty in LaTeX.}\renewcommand\color[2][]{}}\providecommand\includegraphics[2][]{\GenericError{(gnuplot) \space\space\space\@spaces}{Package graphicx or graphics not loaded}{See the gnuplot documentation for explanation.}{The gnuplot epslatex terminal needs graphicx.sty or graphics.sty.}\renewcommand\includegraphics[2][]{}}\providecommand\rotatebox[2]{#2}\@ifundefined{ifGPcolor}{\newif\ifGPcolor
    \GPcolortrue
  }{}\@ifundefined{ifGPblacktext}{\newif\ifGPblacktext
    \GPblacktexttrue
  }{}\let\gplgaddtomacro\g@addto@macro
\gdef\gplbacktext{}\gdef\gplfronttext{}\makeatother
  \ifGPblacktext
\def\colorrgb#1{}\def\colorgray#1{}\else
\ifGPcolor
      \def\colorrgb#1{\color[rgb]{#1}}\def\colorgray#1{\color[gray]{#1}}\expandafter\def\csname LTw\endcsname{\color{white}}\expandafter\def\csname LTb\endcsname{\color{black}}\expandafter\def\csname LTa\endcsname{\color{black}}\expandafter\def\csname LT0\endcsname{\color[rgb]{1,0,0}}\expandafter\def\csname LT1\endcsname{\color[rgb]{0,1,0}}\expandafter\def\csname LT2\endcsname{\color[rgb]{0,0,1}}\expandafter\def\csname LT3\endcsname{\color[rgb]{1,0,1}}\expandafter\def\csname LT4\endcsname{\color[rgb]{0,1,1}}\expandafter\def\csname LT5\endcsname{\color[rgb]{1,1,0}}\expandafter\def\csname LT6\endcsname{\color[rgb]{0,0,0}}\expandafter\def\csname LT7\endcsname{\color[rgb]{1,0.3,0}}\expandafter\def\csname LT8\endcsname{\color[rgb]{0.5,0.5,0.5}}\else
\def\colorrgb#1{\color{black}}\def\colorgray#1{\color[gray]{#1}}\expandafter\def\csname LTw\endcsname{\color{white}}\expandafter\def\csname LTb\endcsname{\color{black}}\expandafter\def\csname LTa\endcsname{\color{black}}\expandafter\def\csname LT0\endcsname{\color{black}}\expandafter\def\csname LT1\endcsname{\color{black}}\expandafter\def\csname LT2\endcsname{\color{black}}\expandafter\def\csname LT3\endcsname{\color{black}}\expandafter\def\csname LT4\endcsname{\color{black}}\expandafter\def\csname LT5\endcsname{\color{black}}\expandafter\def\csname LT6\endcsname{\color{black}}\expandafter\def\csname LT7\endcsname{\color{black}}\expandafter\def\csname LT8\endcsname{\color{black}}\fi
  \fi
    \setlength{\unitlength}{0.0500bp}\ifx\gptboxheight\undefined \newlength{\gptboxheight}\newlength{\gptboxwidth}\newsavebox{\gptboxtext}\fi \setlength{\fboxrule}{0.5pt}\setlength{\fboxsep}{1pt}\begin{picture}(4520.00,2820.00)\gplgaddtomacro\gplbacktext{\csname LTb\endcsname \put(616,931){\makebox(0,0)[r]{\strut{}$10^{-6}$}}\csname LTb\endcsname \put(616,1490){\makebox(0,0)[r]{\strut{}$10^{-4}$}}\csname LTb\endcsname \put(616,2049){\makebox(0,0)[r]{\strut{}$10^{-2}$}}\csname LTb\endcsname \put(1521,448){\makebox(0,0){\strut{}$500$}}\csname LTb\endcsname \put(3282,448){\makebox(0,0){\strut{}$1500$}}}\gplgaddtomacro\gplfronttext{\csname LTb\endcsname \put(2445,142){\makebox(0,0){\strut{}$h$}}}\gplbacktext
    \put(0,0){\includegraphics[width={226.00bp},height={141.00bp}]{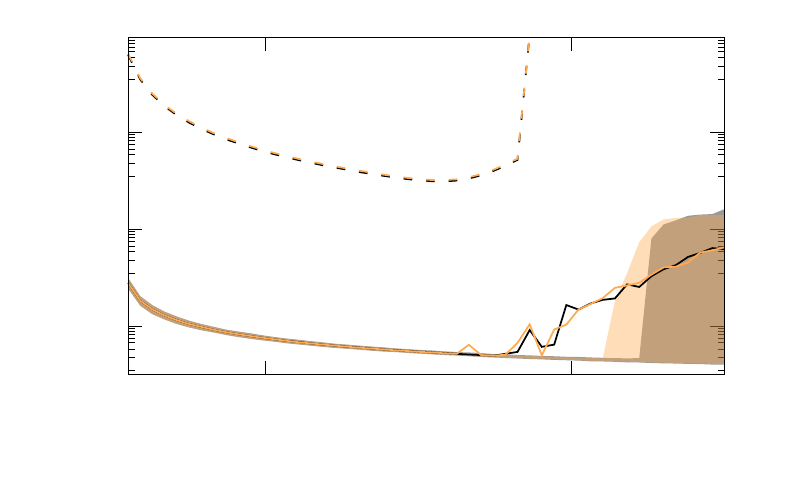}}\gplfronttext
  \end{picture}\endgroup
     \caption{The solid lines are the mean values of the risk functional
      (black) and the cross-validation score (orange). The transparent tubes
      represent 90\% of all outcomes. The dashed lines are our theoretical
      bounds for these regions.}
  \end{subfigure}
  \begin{subfigure}{0.50\linewidth}
    \centering
    \begingroup
  \makeatletter
  \providecommand\color[2][]{\GenericError{(gnuplot) \space\space\space\@spaces}{Package color not loaded in conjunction with
      terminal option `colourtext'}{See the gnuplot documentation for explanation.}{Either use 'blacktext' in gnuplot or load the package
      color.sty in LaTeX.}\renewcommand\color[2][]{}}\providecommand\includegraphics[2][]{\GenericError{(gnuplot) \space\space\space\@spaces}{Package graphicx or graphics not loaded}{See the gnuplot documentation for explanation.}{The gnuplot epslatex terminal needs graphicx.sty or graphics.sty.}\renewcommand\includegraphics[2][]{}}\providecommand\rotatebox[2]{#2}\@ifundefined{ifGPcolor}{\newif\ifGPcolor
    \GPcolortrue
  }{}\@ifundefined{ifGPblacktext}{\newif\ifGPblacktext
    \GPblacktexttrue
  }{}\let\gplgaddtomacro\g@addto@macro
\gdef\gplbacktext{}\gdef\gplfronttext{}\makeatother
  \ifGPblacktext
\def\colorrgb#1{}\def\colorgray#1{}\else
\ifGPcolor
      \def\colorrgb#1{\color[rgb]{#1}}\def\colorgray#1{\color[gray]{#1}}\expandafter\def\csname LTw\endcsname{\color{white}}\expandafter\def\csname LTb\endcsname{\color{black}}\expandafter\def\csname LTa\endcsname{\color{black}}\expandafter\def\csname LT0\endcsname{\color[rgb]{1,0,0}}\expandafter\def\csname LT1\endcsname{\color[rgb]{0,1,0}}\expandafter\def\csname LT2\endcsname{\color[rgb]{0,0,1}}\expandafter\def\csname LT3\endcsname{\color[rgb]{1,0,1}}\expandafter\def\csname LT4\endcsname{\color[rgb]{0,1,1}}\expandafter\def\csname LT5\endcsname{\color[rgb]{1,1,0}}\expandafter\def\csname LT6\endcsname{\color[rgb]{0,0,0}}\expandafter\def\csname LT7\endcsname{\color[rgb]{1,0.3,0}}\expandafter\def\csname LT8\endcsname{\color[rgb]{0.5,0.5,0.5}}\else
\def\colorrgb#1{\color{black}}\def\colorgray#1{\color[gray]{#1}}\expandafter\def\csname LTw\endcsname{\color{white}}\expandafter\def\csname LTb\endcsname{\color{black}}\expandafter\def\csname LTa\endcsname{\color{black}}\expandafter\def\csname LT0\endcsname{\color{black}}\expandafter\def\csname LT1\endcsname{\color{black}}\expandafter\def\csname LT2\endcsname{\color{black}}\expandafter\def\csname LT3\endcsname{\color{black}}\expandafter\def\csname LT4\endcsname{\color{black}}\expandafter\def\csname LT5\endcsname{\color{black}}\expandafter\def\csname LT6\endcsname{\color{black}}\expandafter\def\csname LT7\endcsname{\color{black}}\expandafter\def\csname LT8\endcsname{\color{black}}\fi
  \fi
    \setlength{\unitlength}{0.0500bp}\ifx\gptboxheight\undefined \newlength{\gptboxheight}\newlength{\gptboxwidth}\newsavebox{\gptboxtext}\fi \setlength{\fboxrule}{0.5pt}\setlength{\fboxsep}{1pt}\begin{picture}(4520.00,2820.00)\gplgaddtomacro\gplbacktext{\csname LTb\endcsname \put(616,869){\makebox(0,0)[r]{\strut{}$10^{-8}$}}\csname LTb\endcsname \put(616,1520){\makebox(0,0)[r]{\strut{}$10^{-5}$}}\csname LTb\endcsname \put(616,2171){\makebox(0,0)[r]{\strut{}$10^{-2}$}}\csname LTb\endcsname \put(1521,448){\makebox(0,0){\strut{}$500$}}\csname LTb\endcsname \put(3282,448){\makebox(0,0){\strut{}$1500$}}}\gplgaddtomacro\gplfronttext{\csname LTb\endcsname \put(2445,142){\makebox(0,0){\strut{}$h$}}}\gplbacktext
    \put(0,0){\includegraphics[width={226.00bp},height={141.00bp}]{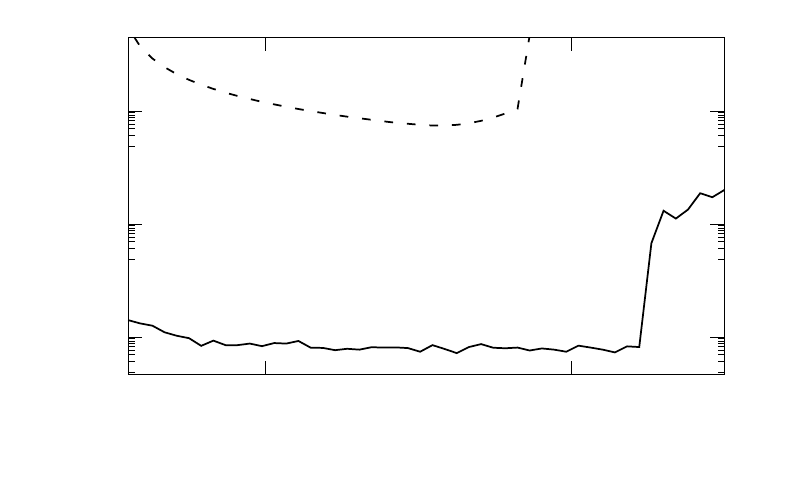}}\gplfronttext
  \end{picture}\endgroup
     \caption{The solid line is the 90\%-quantile of the differences
      between the risk functional and cross-validation score. The dashed line
      is our theoretical bound for this quantity.\vspace{3ex}}
  \end{subfigure}
  \caption{Numerical example on $\mathbb T$}
  \label{fig:bound}
\end{figure}

In Figure~\ref{fig:bound} (c) and (d) our theoretical bounds rise rapidly at $h\approx 1500$ which coincides with the beginning of instability in the computation of Shepard's model.

 \section{Conclusion}

In this paper we presented a framework for obtaining bounds for the difference of cross-validation score and risk functional with high probability.
This speaks for the use of cross-validation in parameter choice questions.
In contrast to most previous results, we obtain a pre-asymptotic statement.

Along the way we proved concentration inequalities for the cross-validation score and risk functional, respectively.
Connecting their expected values, we were able to combine both concentration inequalities and build a machinery to bound their difference with high probability.
All those results are based on uniform bounds of the reconstruction method, which must hold in a subset of all possible samples.
Estimates of this type are broadly available in learning theory.

For demonstration purposes we used Shepard's model on the one-dimensional torus with a rather simple bound of the uniform error.
Numerical examples with a fast implementation support our results.
 
\section*{Acknowledgments}
Felix Bartel acknowledges funding by the European Social Fund (ESF), Project ID 100367298.

\bibliographystyle{abbrv}

\end{document}